\documentclass[a4paper, 10pt]{article}
\usepackage{srcltx,graphicx,epstopdf}
\usepackage{amsmath, amssymb, amsthm}
\usepackage{color}
\usepackage{xcolor}
\usepackage{multirow}
\usepackage{subfig} 
\usepackage{hyperref}
\usepackage[margin=1in]{geometry}
\usepackage{cleveref}
\crefname{section}{Section}{Sections}
\crefname{subsection}{Subsection}{Subsections}
\crefname{appendix}{Appendix}{Appendix}
\crefname{figure}{Figure}{Figures}
\crefname{table}{Table}{Tables}
\crefname{property}{Property}{Properties}
\crefname{theorem}{Theorem}{Theorem}
\usepackage{booktabs}
\usepackage{threeparttable}
\usepackage{enumerate}
\usepackage[normalem]{ulem}
\graphicspath{{images/}}

\newtheorem{proposition}{Proposition}

{\theoremstyle{remark} }

\newcommand\dd{\mathrm{d}}
\newcommand\pd[2]{\dfrac{\partial {#1}}{\partial {#2}}}

\def\bF{\boldsymbol{F}}
\def\bbS{\mathbb{S}}
\def\bbR{\mathbb{R}}
\def\bsOmega{\boldsymbol{\Omega}}
\def\bx{\boldsymbol{x}}
\def\bn{\boldsymbol{n}}
\def\ie{\mathrm{e}}
\numberwithin{equation}{section}

\title {Unified Gas-Kinetic Particle Method for Frequency-dependent Radiation Transport}

\author{Weiming Li \thanks{Institute of Applied Physics and Computational Mathematics, Beijing,
    China, email: {\tt li\_weiming@iapcm.ac.cn}},
 ~~Chang Liu \thanks{Corresponding author. Institute of Applied Physics and Computational Mathematics, Beijing,
    China, email: {\tt liuchang@iapcm.ac.cn}},
    ~~Peng Song \thanks{Institute of Applied Physics and Computational Mathematics,
   Beijing, China and HEDPS, Center for Applied Physics and Technology, College of Engineering, Peking University, Beijing, China, email: \tt{song\_peng@iapcm.ac.cn}}}

\begin{document}
\maketitle

\begin{abstract}

This paper proposes a unified gas-kinetic particle (UGKP) method for the frequency-dependent photon transport process.
The photon transport is a typical multiscale process governed by the nonlinear radiative transfer equations (RTE). 
The flow regime of photon transport varies from the ballistic regime to the diffusive regime with respect to optical depth and photon frequency. 
The UGKP method is an asymptotic preserving (AP) scheme and a regime adaptive scheme. 
The teleportation error is significantly reduced, and the computational efficiency is remarkably improved in the diffusive regime.
Distinguished from the standard multigroup treatment, the proposed UGKP method solves the frequency space in a non-discretized way. Therefore the Rosseland diffusion system can be precisely preserved in the optically thick regime.
Taking advantage of the local integral solution of RTE, the distribution of the emitted photon can be constructed from its macroscopic moments. 
The Monte Carlo particles in the UGKP method need only be tracked before their first collision events, and a re-sampling process is performed to close the photon distribution for each time step. 
The large computational cost of excessive scattering events can be saved, especially in the optically thick regime. 
The proposed UGKP method is implicit and removes the light speed constraint on the time step. 
The particle tracking approach combining the implicit formulation makes the proposed UGKP method an efficient solution algorithm for frequency-dependent radiative transfer problems. 
We demonstrate with numerical examples the capability of the proposed multi-frequency UGKP method.

  \vspace*{4mm}
  \noindent {\bf key word:} radiative transfer equation; frequency-dependent; asymptotic
  preserving; regime adaptive; Multiscale method

\end{abstract}

\section{Introduction}
Radiation transport is the essential energy transfer process in astrophysics and high-energy-density physics \cite{MRE1,MRE2,MRE3,MRE4}.
It is described by a system of nonlinear equations, which includes a kinetic equation modeling the transport of photons and an equation for material energy evolution. 
As the kinetic equation depends on time, space, angular, and frequency coordinates, it is generally seven-dimensional, making it computationally expensive. 
To decrease the dimensionality, 
simplified models are sometimes used. 
In optically thick regions where the mean free path is small compared to the characteristic length, the propagation of photons is
essentially a diffusion process. 
In such media, the Rosseland diffusion equation is naturally used as the asymptotic limit of the radiative transfer equations \cite{Larsen1983}. 
Its flux-limited form \cite{levermore1981flux} is also used \cite{whitehouse2004smoothed,mayer2007fragmentation}.
Another approach for reducing the dimensionality is to use the moment models, which replace directed radiation with a direction-independent radiative flux\cite{sherman1967moment,li2021nonlinear}. 

For decades, there has been a continued effort to develop accurate and efficient numerical schemes to solve radiative transfer equations.
Numerical schemes can be categorized as deterministic methods and stochastic methods.
One of the most commonly used deterministic methods is the discrete ordinate method ($S_N$), which uses a quadrature rule to discretize the angular variable. 
The most popular stochastic method is the implicit Monte Carlo (IMC) method originally proposed in \cite{fleck1971}. 
The IMC explicitly calculates the physical events of photon absorption, scattering, advection, and emission by utilizing a linearization of the re-emission physics. 
Deterministic methods suffer from excessive numerical dissipation in optically thick regimes. 
High-order schemes \cite{maginot2016high,xiong2022high} have been developed to improve accuracy.
The flow regime of photon transport changes significantly with respect to material opacity and photon frequency. 
Much effort has been put into developing asymptotic preserving (AP) schemes to simulate multiscale photon transport problems. 
The AP schemes are consistent and stable for the limiting diffusion equation as the mean free path goes to zero and has uniform error estimates with respect to the mean free path \cite{jin1999efficient,xu2010unified,jin2010asymptotic}. 
A class of high-order asymptotic preserving schemes for nonlinear gray radiative transfer equations have been developed under the DG-IMEX 
framework \cite{xiong2022high}. 
For the Monte Carlo method, the excessive numerical dissipation results in the heat wave propagating faster than in reality in 
opaque material. Such error is usually described by the term teleportation error and can be reduced by source tilting \cite{densmore2011asymptotic, shi2021continuous}. 
Recently, an implicit semi-analog Monte Carlo method \cite{poette2020new, steinberg2022multi} is developed, which eliminates teleportation
error by sampling emitting photons only from existing material particles.   

There have also been continued efforts to improve the efficiency of numerical schemes in solving the radiative transfer equation. 
Both deterministic methods and the traditional Monte Carlo method suffer from low efficiency in optically thick regimes. 
The $S_N$ method often uses source iteration for inverting the discretized system, but it often has a poor convergence rate for optically thick regions \cite{adams2002fast}.
Methods for speeding up the convergence of the $S_N$ method in diffusive regimes include the diffusion synthetic acceleration (DSA) method,  which is essentially a preconditioning technique solving a diffusion problem in every iteration \cite{Morel1982Synthetic,palii2020on}. 
Other approaches include multiscale acceleration techniques such as the multi-scale high order/low order (HOLO) algorithm \cite{chacon2017multiscale} and the unified gas-kinetic scheme (UGKS) \cite{xu2010unified}. 
The HOLO-type algorithm solves a high-order (HO) microscopic system and a low-order (LO) macroscopic system through bi-directional nonlinear
coupling. 
It utilizes nonlinear elimination to enslave the HO component in the LO one so that various preconditioning strategies can be employed to accelerate the resulting residual system. 
The UGKS couples transport and collision processes using a multiscale flux function obtained from the integral solution of the 
kinetic model equation. 
Initially developed in the rarefied gas dynamics \cite{xu2010unified}, the UGKS has been applied to a wide range of transport processes \cite{liu2021unified,liu2019limitation}, including photon transport \cite{sun2015asymptotic1,sun2015asymptotic2}.
Other approaches have also been developed to speed up convergence for iteration schemes in the field of rarefied gas dynamics, such as the general synthetic iterative scheme (GSIS) \cite{wu2017431, su2020} and the discrete unified gas-kinetic scheme (DUGKS) \cite{guo2013dugks, guo_progress_DUGKS}.
In optically thick regions, the Monte Carlo method also suffers from low efficiency due to the need to compute a large amount of effective scattering.
In order to improve efficiency, various hybrid methods are proposed, where a Monte Carlo solver is employed in optically thin regions, 
and a diffusion equation is solved in optically thick regions \cite{wollaeger2013radiation, keady2017improved}. 
The unified gas-kinetic particle method (UGKP) and unified gas-kinetic wave-particle method (UGKWP) have been proposed \cite{liweiming2020, shi2020}. 
These are multiscale methods that consistently couple the evolution of macroscopic and microscopic equations and do not need to employ multiple effective scattering events. 
Therefore, the computation cost could be significantly reduced in optically thick regimes. 
Most of the above-mentioned method for improving accuracy and efficiency are originally designed for the gray approximation to the transport equation, and some of them has been extended to frequency-dependent cases.



In our recent work \cite{liweiming2020}, we apply the UGKP and UGKWP framework proposed in \cite{liu2019} to gray radiative transfer. 
Both methods use a mixture of discrete particles and analytical functions to represent the specific intensity's angular dependency. 
This formulation works well for problems where opacity weakly depends on frequency. 
However, in real-world applications, we deal with situations where opacity strongly depends on frequency. 
In this work, we present the extension of the UGKP method to multiple frequency scenarios.
The main highlights of our work are : (i) our method targets frequency-dependent photon transport, where opacities vary widely in both space and frequency;
(ii) in optically thick regions,
our method converges to an implicit discretization of the limiting Rosseland diffusion equation, enabling us to take large time steps; 
(iii) radiation and material energy are fully coupled in our new scheme. We do not employ either the linearization of the implicit Monte Carlo method or the splitting between effective scattering and absorption, as utilized in our previous paper \cite{liweiming2020}.

The remainder of the paper is structured as follows. 
We first introduce the basics of the radiative transfer equation and its diffusion limit in Section 2, then present our UGKP method in Section 3.  
A numerical analysis of the proposed UGKP method is presented in Section 4.  
We then present numerical results, including the Marshak wave and hohlraum examples, to show the capabilities of this method in Section 5. 

\section{Radiative transfer equation}

The equations of radiative transfer describe the interaction between the radiation field and matter. It considers the processes of absorption, 
scattering, and emission of photons. Photons are treated as point, massless particles, and their wave-like behavior is omitted. 

We consider the frequency-dependent radiative transfer system of equations. Scattering is currently omitted and will be considered in future work.
In the absence of hydrodynamic motion and heat conduction, and under the assumption of local thermodynamic equilibrium (LTE) with induced
processes taken into account, the system of radiative transfer equations takes the form \cite{Pomraning1973}:
\begin{subequations}\label{eq:rte}
\begin{align}
\label{eq:I_eq}
& \dfrac{\epsilon^2}{c}\dfrac{\partial I}{\partial t}+\epsilon\bsOmega\cdot\nabla I =-\sigma I + \sigma B(\nu, T),\\[3mm]
\label{eq:T_eq}
& \epsilon^2 C_v\dfrac {\partial T}{\partial t} = \int_0^\infty \left(\sigma \int_{\bbS^2} I(\bsOmega)\dd\bsOmega - \sigma 4 \pi B(\nu,T)\right) \dd\nu.
\end{align}
\end{subequations}
$\epsilon$ is a nondimensionalization parameter defined as the ratio between the typical 
mean free path of photons and the typical length scale of the physical system.  The specific intensity of photons, $I(t,\bx,\bsOmega,\nu)$, is a 
function of time $t\in\bbR^+$, physical space coordinates $\bx\in D_x \subset \bbR^D$, angular direction $\bsOmega \in \bbS^2$ and frequency $\nu \in \bbR^+$, 
with
\begin{equation}\label{eq:Planck}
B(\nu,T) = \dfrac{2 h \nu^3}{c^2}\dfrac{1}{\ie^{h\nu / kT}-1},
\end{equation}
satisfying
\begin{equation}
\int_0^\infty \int_{\bbS^2} B(\nu, T)\dd\bsOmega \dd\nu = acT^4.
\end{equation}
In equation \eqref{eq:rte}, $C_v$ is the specific heat capacity of 
the background material, $\sigma$ is the absorption/emission coefficient which takes into account induced emission, 
$-\sigma I$ characterizes the absorption process, and $\sigma B(\nu, T)$ 
describes the emission of photons under the assumptions of local thermodynamic equilibrium. In equation \eqref{eq:Planck}, $h$ is the Planck constant, 
$c$ is the speed of light, $k$ is the Boltzmann constant, and $a$ is the radiation constant satisfying
\begin{equation}
a  = \dfrac{8 \pi^5 k^4}{15 h^3 c^3}.
\end{equation}
Though limited in scope, equation \eqref{eq:rte} is nevertheless
useful in many applications and has been studied widely in literature \cite{larsen1988grey, sun2015asymptotic2,Hammer2019}.

We impose the initial condition
\begin{equation}\label{eq:initial_cond}
I(0,\bx,\bsOmega,\nu) = I_0(\bx,\bsOmega,\nu),\quad T(0,\bx) = T_0(\bx),
\end{equation}
where $I_0$ is a specified function relying on space, angle, and frequency, while the initial temperature $T_0$ is a specified function of space.

An important case of the boundary condition is the inflow boundary condition. 
It is sufficient to specify the 
specific intensity at all points on the surface in the incoming direction, which implies that 
\begin{equation}\label{eq:boundary_cond_inflow}
I(t,\bx_{bd},\bsOmega,\nu) = I_{bd}(t,\bx_{bd},\bsOmega,\nu), \bn \cdot \bsOmega < 0,
\end{equation}
where $\bn$ is the outer normal direction for $\bx \in \partial D_x$. 

Equation \eqref{eq:rte}, together with the initial condition \eqref{eq:initial_cond} and boundary condition such as \eqref{eq:boundary_cond_inflow} 
completely specify 
the system of radiative transfer equation.

The high dimensionality and nonlinearity of equation \eqref{eq:rte} make it expensive to discretize. Therefore, some limiting regimes are identified.
 Two important limiting regimes are often 
considered.  The first regime is the free-streaming limit, where $\sigma = 0$, which models optically thin regions.
In this case, photons transport freely without interaction with 
matter, and the material temperature remains constant in time. For the second regime, the assumptions are that the mean free path is small compared to the characteristic scale, and the solutions vary slowly in space and time  \cite{Larsen1983}. It models the optically thick regions, and
is sometimes called the parabolic scaling
limit \cite{tang2021} and is obtained by sending $\epsilon$ to $0$. It was shown in \cite{Larsen1983} that 
as $\epsilon$ approaches $0$, away from the boundary and initial layers, the specific intensity satisfy
\begin{equation}\label{eq:I_ap}
I(t,\bx,\bsOmega,\nu) = B(\nu,T)-\dfrac{\epsilon}{\sigma}\bsOmega\cdot\nabla B(\nu,T) + \mathcal{O}(\epsilon^2),
\end{equation}
from which one could obtain Fick's law
\begin{equation}\label{eq:Fick_law}
\int_0^\infty \int_{\bbS^2} \bsOmega I(t,\bx,\bsOmega,\nu) \dd \bsOmega \dd \nu = \bF = -\dfrac{\epsilon ac}{3\sigma_R}\nabla T^4.
\end{equation}
In equation \eqref{eq:Fick_law}, 
\begin{equation}\label{eq:sigmaR}
\sigma_R = \dfrac{\displaystyle{\int_0^\infty \pd{B}{T} \dd\nu}}{\displaystyle{\int_0^\infty \dfrac{1}{\sigma}\pd{B}{T} \dd\nu}}
\end{equation}
is the Rosseland mean cross-section.
Integrating \eqref{eq:I_eq} and adding it to \eqref{eq:T_eq}, and then applying Fick's law yields the Rosseland 
diffusion limit equation
\begin{equation}\label{eq:rosseland_diffusion_limit}
\pd{}{t}\left(U_m + a T^4\right) = \nabla \cdot \left(\dfrac{a c}{3 \sigma_R}\nabla T^4\right),
\end{equation}
where $C_v = \pd{U_m}{T}$.
As equation \eqref{eq:rosseland_diffusion_limit}  is the asymptotic limit of the radiative transfer equation as $\epsilon\rightarrow0$, 
it is natural to require that asymptotic preserving schemes limit  to a discretization of equation \eqref{eq:rosseland_diffusion_limit} in optically thick regions.

To our knowledge, most asymptotic preserving schemes for the frequency-dependent radiative transfer equation in current literature such as \cite{sun2018asymptotic}
are based on multigroup discretization of frequency as described in \cite{Pomraning1973}. The multigroup method does not treat frequency as a continuous
variable. Instead, the frequency interval is truncated and divided into $G$ frequency groups $\nu \in \cup [\nu_g, \nu_{g+1}]$, and the
$g$th group specific intensity is defined as
\begin{equation}
I_g(t,\bx,\bsOmega) = \int_{\nu_g}^{\nu_{g+1}} I(t,\bx,\bsOmega,\nu) \dd\nu.
\end{equation}
Equation \eqref{eq:I_eq} is integrated over $\forall g$, $\nu\in\cup[\nu_g,\nu_{g+1}]$ to obtain
\begin{equation}\label{eq:rte_multigroup}
\dfrac{\epsilon^2}{c}\dfrac{\partial I_g}{\partial t}+\epsilon\bsOmega\cdot\nabla I_g =-\sigma_g I_g + \int_{\nu_g}^{\nu_{g+1}} \sigma(\nu) B(\nu, T) \dd\nu,
\end{equation}
where $\sigma_g$ is taken as either a group Rosseland or group Planck mean. Although the Planck mean absorption coefficient  is 
more appropriate for optically thin cases,
the Rosseland mean is the more appropriate one for optically thick limits. As pointed out in \cite{Pomraning1973}, these two means differ greatly for realistic
absorption coefficients, and either is only strictly appropriate in its limiting circumstances. In practice, one or the other is generally used, based on which
asymptotic preserving schemes for frequency-dependent radiative transfer equation in current literature is developed \cite{sun2018asymptotic}. 
We aim to develop a scheme that 
asymptotically preserves the exact Rosseland diffusion limit for optically thick regions, and also preserves the correct free-streaming limit.

%
%

\section{Multi-frequency unified gas-kinetic particle method}
In our previous work \cite{liweiming2020}, we first presented a unified gas-kinetic particle (UGKP) method for the gray equation of transfer. 
A similar method was presented and discussed in \cite{shi2020}.
In this section, we extend the  UGKP method to frequency-dependent cases. As was for the gray equation of transfer, the evolution of microscopic simulation particle is coupled with the evolution of macroscopic energy in the UGKP method. In the following, the scheme for the evolution of microscopic particles and macroscopic energy field will be introduced. For simplicity, the method will be presented for the two-dimensional case on the Cartesian mesh. Its 
extension to 3D is straightforward. Extension to general quadrilateral mesh is part of ongoing work.

\subsection{Particle evolution}\label{sec:UGKP_micro}
The microscopic evolution of the UGKP method uses the particle-based Monte Carlo solver to update the specific 
intensity in equation \eqref{eq:rte}. Assume the total number of particles to be $N_p$. Each particle $p_k = (w_k,\bx_k(t),\bsOmega_k,\nu_k)$
is characterized by its weight $w_k$, 
position $\bx_k$, velocity angle $\bsOmega_k$, and frequency $\nu_k$, $1\leq k\leq N_p$.  Particles are
generated from the 
initial condition, boundary source, and re-sampling at the end of each time step.  Initially, particles are sampled such that particle position,
 velocity angle, and frequency follow the distribution $I(t_0,\bx,\bsOmega,\nu)$, i.e.
 \begin{equation}
 I(t_0,\bx,\bsOmega,\nu) = \sum\limits_{k=1}^{N_p} \dfrac{w_k}{V} \delta(\bx-
 \bx_k(t_0))\delta(\bsOmega-\bsOmega_k)\delta(\nu-\nu_k),
 \end{equation}
 where $V$ is the volume of the cell in wich the particle resides. At any
 instant $t$, $w_k$ satisfy
 \begin{equation}
 \sum\limits_{k=1}^{N_p}w_k = \int_{D_x}\int_0^\infty \int_{\bbS^2} I(t,\bx,
 \bsOmega,\nu) \dd\bsOmega \dd\nu \dd\bx,
 \end{equation}
 where $D_x$ is the computation domain.


%

Different from the UGKP method for the gray equation of transfer
as discussed in \cite{liweiming2020}, we do not adopt the strategy of using the transformation of the IMC method and splitting for evolving the 
specific intensity. Instead, we 
evolve particle information by the integral solution of \eqref{eq:rte} which takes the form:
\begin{equation}\label{eq:integral_solution}
\begin{split}
I(t,\bx,\bsOmega,\nu) = & \ie^{-c \sigma(\nu,T) t /\epsilon^2} I\left(0,\bx-\dfrac{c\bsOmega}{\epsilon} t ,\bsOmega,\nu\right) \\[3mm]
& + \int_{0}^t \ie^{-c \sigma(\nu,T)(t-s)/\epsilon^2}
 \dfrac{c \sigma(\nu,T)}{\epsilon^2}  B\left(\nu, T\left(s,\bx-\dfrac{c \bsOmega}{\epsilon}(t-s)\right)\right)\dd s.
 \end{split}
\end{equation}
A second-order approximation of equation \eqref{eq:integral_solution} gives
\begin{equation}\label{eq:integral_solution_approx}
\begin{split}
I(t,\bx,\bsOmega,\nu) =  & \ie^{-c \sigma(\nu,T) t /\epsilon^2} I\left(0,\bx-\dfrac{c\bsOmega}{\epsilon} t ,\bsOmega,\nu\right) +
C_1(t)B(\nu,T(t,\bx))\\[3mm]
&+C_2(t)\pd{B}{T}(\nu,T(t,\bx))\pd{T}{t}(t,\bx)+C_3(t)\pd{B}{T}(\nu,T(t,\bx))\bsOmega\cdot\nabla T(t,\bx),
\end{split}
\end{equation}
where
\begin{equation}\label{eq:coef_C1}
C_1(t)= 1 - \ie^{-c \sigma t / \epsilon^2},
\end{equation}
\begin{equation}
C_2(t)=\dfrac{-\epsilon^2+\ie^{-c \sigma t/\epsilon^2}(\epsilon^2+c\sigma t)}{c\sigma},
\end{equation}
\begin{equation}
C_3(t)=\dfrac{-\epsilon^2+\ie^{-c \sigma t/\epsilon^2}(\epsilon^2+c\sigma t)}{\epsilon\sigma}.
\end{equation}
Equation \eqref{eq:integral_solution} states that for a time period $t$, the probability of particle free-streaming is $\exp(- c \sigma(\nu,T) t / \epsilon^2)$. 
Consequently, the particle has a probability 
of $1-\exp(-c \sigma(\nu,T) t / \epsilon^2)$ to be absorbed and re-emitted.
The distribution of the re-emitted particles is analytically given in Eq. \eqref{eq:integral_solution} as the second integral term.
UGKP does not track the re-emission process for the absorbed particles. 
Instead, the absorbed particles are re-sampled from the approximated distribution \eqref{eq:integral_solution_approx} at the end of the time step to close the distribution.  
Compared to the traditional Monte Carlo method, we do not need to  resolve each absorption  and
 re-emission process, so multiple absorption and emission events can happen for each particle within a time step, and therefore 
the UGKP method achieves high efficiency. Compared to the UGKP method we employed for the gray equation
of transfer,  the total cross-section
instead of the effective scattering cross-section is used to determine free-streaming particles. Therefore, for cases where the absorption coefficient
is significant but the effective scattering term does not dominate, fewer particles are tracked in our new strategy, and higher efficiency is achieved. 

At any time step $t_n$, we define
\begin{equation}
\mathcal{P}^n_{\mathrm{D}}=\{p_k|\bx_k(t_n) \in D_x\}.
\end{equation}
We denote by $\mathcal{P}_{\mathrm{in}}^n$ the set of all particles flowing
into the domain $D_x$ within the time step $[t_{n-1},t_n]$, by $\mathcal{P}_{\mathrm{out}}^n$ all particles flowing out of $D_x$ within the
same period, and by $\mathcal{P}_{\mathrm{C}}^n$ the set of all particles 
for which absorption event occurs between $t_{n-1}$ and $t_n$. $\mathcal{P}_{\mathrm{r}}^n$ is used to denote the set of all re-sampled particles at $t_n$. The re-sampling process will be discussed specifically in
Section \ref{sec:re-sample}. The UGKP method satisfies
\begin{equation}
\mathcal{P}^n_{\mathrm{D}} = \mathcal{P}^{n-1}_{\mathrm{D}}+\mathcal{P}^n_{\mathrm{in}}
-\mathcal{P}^n_{\mathrm{out}}-\mathcal{P}^n_{\mathrm{C}}
+\mathcal{P}^n_{\mathrm{r}}.
\end{equation}

For evolving the trajectory of each Monte Carlo particle, we consider  three possible types of events  that could occur as 
each particle travels through the background medium: absorption, boundary crossing, and survival at the end of the time step. 
The distance traveled by a particle before 
absorption occurs satisfy
\begin{equation}\label{eq:col_dis}
d_{\mathrm{COL}} = \dfrac{\epsilon^2}{\sigma} |\ln \xi |,
\end{equation}
where $\xi$ is a random number that follows a uniform distribution on $(0,1)$. When an absorption event occurs, a particle is removed from the 
system. Such particles are then re-sampled from the integral solution at the end of the time step.  The distance to boundary crossing $d_\mathrm{B}$ satisfy
\begin{equation}
\bx_{B} - \bx_{p} = d_{\mathrm{B}} \bsOmega,
\end{equation}
where $\bx_p$ is the original location of the tracked particles, $\bx_B$ is the location of the cell interface it crosses, and $\bsOmega$ is the 
direction of particle flight. The distance to the census,
in which case a particle survives during the whole time step, is 
\begin{equation}
d_{\mathrm{CEN}} = c(t_{n+1}-t),
\end{equation}
where $t$ is the current time of the particle considered. The distance each particle travel during one cycle of the tracking process is 
$d =\text{min}\{d_{\mathrm{COL}}, d_{\mathrm{B}}, d_{\mathrm{CEN}}\}$. The tracking process for each particle is continued until the particle has leaked out of
 the system, been absorbed, or reaches the end of the specific time step.

Under our formulation,  at any specified $t_n$,
particle position, velocity angle, and frequency follow
the distribution $I(t_n,\bx,\bsOmega,\nu)$. Therefore, all information on the specific intensity's reliance on frequency is kept and
the specific intensity is not group-averaged.


\subsection{Macroscopic evolution}
As described in Section \ref{sec:UGKP_micro}, the distribution of re-sampled particles rely on material temperature, 
therefore we evolve it by solving a set of macroscopic equation coupling it to radiation energy. Define
\begin{equation}
\langle g \rangle =  \int_0^\infty \int_{\bbS^2} g \dd \bsOmega \dd\nu,
\quad
\rho = \langle I \rangle.
\end{equation}
 Integrate equation \eqref{eq:I_eq} with respect to $\bsOmega\in \bbS^2$ and $\nu \in \bbR^+$, 
and combining the result with \eqref{eq:T_eq},
we obtain the coupled macroscopic equations:
\begin{subequations}\label{eq:macro_rte}
\begin{align}
\label{eq:rho_macro_eq}
& \dfrac{\epsilon^2}{c}\dfrac{\partial \rho}{\partial t}+\epsilon \nabla \cdot \langle \bsOmega I \rangle =-\langle \sigma I \rangle +
4 \pi \int_0^\infty \sigma B(\nu, T) \dd \nu,\\[3mm]
\label{eq:T_macro_eq}
& \epsilon^2 C_v\dfrac {\partial T}{\partial t} = \langle \sigma I \rangle - 4 \pi\int_0^\infty \sigma  B(\nu,T) \dd\nu.
\end{align}
\end{subequations}

%
%
%
Let 
\begin{equation}\label{eq:Ur}
U_r = a c T^4,
\end{equation}
and
\begin{equation}\label{eq:normalized_planck}
b(\nu, T) = \dfrac{4\pi}{U_r} B(\nu,T),
\end{equation}
we have
\begin{equation}
\int_0^\infty b(\nu, T) \dd \nu = 1.
\end{equation}
Defining
\begin{equation}\label{eq:sigmaP}
\sigma_p = \int_0^\infty \sigma(\nu,T) b(\nu,T)\dd\nu,
\end{equation}
 and $\beta = 4 a c T^3$, then multiplying \eqref{eq:T_macro_eq} by $\beta$, we obtain 
 \begin{equation}\label{eq:Ur_eq}
\epsilon^2 C_v \pd{U_r}{t} = \beta (\langle \sigma I \rangle - \sigma_p U_r).
\end{equation}
In the UGKP method, we solve the coupled system of equations
\begin{subequations}\label{eq:macro_rte2}
\begin{align}
\label{eq:rho_macro_eq2}
& \dfrac{\epsilon^2}{c}\dfrac{\partial \rho}{\partial t}+\epsilon \nabla \cdot \langle \bsOmega I \rangle =-\langle \sigma I \rangle +
\sigma_p U_r,\\[3mm]
\label{eq:T_macro_eq2}
& \epsilon^2 C_v \pd{U_r}{t} = \beta (\langle \sigma I \rangle - \sigma_p U_r).
\end{align}
\end{subequations}
The macroscopic equation \eqref{eq:macro_rte2} is solved using the finite
 volume method. 
 Denote $i$ and $j$ as the indices for cell $[x_{i-\frac12},x_{i+\frac12}] \times [y_{j-\frac12},y_{j+\frac12}]$. 
 The spatial widths are $\Delta x_i = x_{i+\frac12} - x_{i-\frac12}$ and $\Delta y_j = y_{j+\frac12} - y_{j-\frac12}$ for $x$ and $y$ 
 directions respectively, and the time step size $\Delta t = t_{n+1}-t_n$. We use $\rho^n_{i,j}$ and $T^n_{i,j}$ to approximate the solution $\rho$ and 
 $T$ at time $t_n$ and cell $(i,j)$. Other cell-averaged quantities are
 defined analogously.
Integrating \eqref{eq:rho_macro_eq2} in space  $[x_{i-\frac12},x_{i+\frac12}] \times [y_{j-\frac12},y_{j+\frac12}]$ and
time $[t_n, t_{n+1}]$,  we obtain
\begin{equation}\label{eq:discretize_rho}
\dfrac{\rho^{n+1}_{i,j} - \rho^n_{i,j} }{\Delta t}+\dfrac{\mathcal{F}_{i+\frac12,j}-\mathcal{F}_{i-\frac12,j}}{\Delta x_i}
+\dfrac{\mathcal{G}_{i,j+\frac12}-\mathcal{G}_{i,j-\frac12}}{\Delta y_j} = -\dfrac{c}{\epsilon^2}( S^{n+1}_{i,j} - \sigma^{n+1}_{p,i,j}U^{n+1}_{r,i,j}),
\end{equation}
 while equation \eqref{eq:T_macro_eq2} is discretized by backward Euler in the time variable as
\begin{equation}\label{eq:discretize_Ur}
\epsilon^2 C^{n+1}_{v,i,j} \dfrac{U^{n+1}_{r,i,j}-U^n_{r,i,j}}{\Delta t} = \beta^{n+1}_{i,j}  (S^{n+1}_{i,j}-\sigma^{n+1}_{p,i,j} U^{n+1}_{r,i,j}).
\end{equation}
In both equation \eqref{eq:discretize_rho} and equation \eqref{eq:discretize_Ur}, the source term
\begin{equation}\label{eq:source_define}
S^{n+1}_{i,j} = \dfrac{1}{\Delta x_i \Delta y_j} \int_{x_{i-\frac12}}^{x_{i+\frac12}} \int_{y_{j-\frac12}}^{y_{j+\frac12}} \langle \sigma^{n+1} I^{n+1} \rangle
\dd x \dd y.
\end{equation}
For the situations where $\sigma$ is independent of frequency $\nu$, 
a second-order approximation to the source term gives
\begin{equation}
S^{n+1}_{i,j} = \dfrac{1}{\Delta x_i \Delta y_j} \int_{x_{i-\frac12}}^{x_{i+\frac12}} \int_{y_{j-\frac12}}^{y_{j+\frac12}} \sigma^{n+1}\langle  I^{n+1} \rangle
\dd x \dd y = \sigma^{n+1}_{i,j} \rho^{n+1}_{i,j}.
\end{equation}
However, in general, for frequency-dependent radiation transport, the source term is not closed. From first-order approximation to \eqref{eq:integral_solution_approx}, and taking into consideration consistency
between microscopic and macroscopic variables $I$ and $\rho$,
we postulate a closure by taking
$I^{n+1}$ to follow 
the distribution:
\begin{equation}\label{eq:integral_solution_approx_firstorder}
\begin{split}
I(t,\bx,\bsOmega,\nu) =    & \ie^{-c \sigma(\nu,T) t /\epsilon^2} I\left(0,\bx-\dfrac{c\bsOmega}{\epsilon} t ,\bsOmega,\nu\right) \\[3mm]
& +
\dfrac{1-\ie^{-c \sigma(\nu,T) t/\epsilon^2}}{1 - \int_0^\infty \ie^{-c \sigma(\nu,T) t/\epsilon^2} b(\nu,T)\dd\nu} E^+(t,\bx,\bsOmega) b(\nu,T),
\end{split}
\end{equation}
where
\begin{equation}
E^+(t,\bx,\bsOmega) = \dfrac{1}{4\pi}(\rho-E^{\mathrm{free}}),
\end{equation}
and $E^{\mathrm{free}}$ is the total energy of free-transport particles. The set of all free-transport particles at $t_{n+1}$ within the 
domain $D_x$ is
\begin{equation}
\mathcal{P}^{n+1}_{\mathrm{free}} = \mathcal{P}^n_{\mathrm{D}}+\mathcal{P}^{n+1}_{\mathrm{in}}
-\mathcal{P}^{n+1}_{\mathrm{out}}-\mathcal{P}^{n+1}_\mathrm{C}.
\end{equation}
Therefore, for cell $(i,j)$,
\begin{equation}\label{eq:integral_solution_discrete}
\begin{split}
I^{n+1}_{i,j} = & \ie^{-c\sigma(\nu,T^n)\Delta t /\epsilon^2} I\left(t_n,x_i-\frac{c\Omega_x}{\epsilon}\Delta t, y_j - \frac{c \Omega_y}{\epsilon}\Delta t, \bsOmega,\nu\right)
 +\dfrac{1-\ie^{-c \sigma(\nu,T^{n+1}_{i,j}) \Delta t/\epsilon^2}}{1 - \int_0^\infty \ie^{-c \sigma(\nu,T^{n+1}_{i,j}) \Delta  t/\epsilon^2} b(\nu,T^{n+1}_{i,j})\dd\nu} \\[3mm]
& \times \dfrac{1}{4\pi}\left(\rho^{n+1}_{i,j}-\sum\limits_{p_k\in\mathcal{P}^{n+1}_{\mathrm{free}}} \dfrac{w_k}{\Delta x_i \Delta y_j} \chi_{[x_i,x_{i+1}]\times[y_j,y_{j+1}]} (\bx_k(t_{n+1}) )\right) b(\nu,T^{n+1}_{i,j}).
\end{split}
\end{equation}
 In equation \eqref{eq:integral_solution_discrete}, $w_k$ is the weight of the 
 particle $p_k$,
$\chi$ is the characteristic function, and $\bx_k(t_{n+1})$ is the position of the particle $p_k$ at time $t_{n+1}$. 
Therefore, by plugging \eqref{eq:integral_solution_discrete} 
into \eqref{eq:source_define}, 
the source term is discretized by
\begin{equation}\label{eq:source_close}
\begin{split}
S^{n+1}_{i,j} = & \sum\limits_{p_k\in\mathcal{P}^{n+1}_{\mathrm{free}}} \dfrac{w_k}{\Delta x_i \Delta y_j} \chi_{[x_i,x_{i+1}]\times[y_j,y_{j+1}]}(\bx_k(t_{n+1})) \sigma(\nu_k, T^{n+1}_{i,j})    
\\[3mm]
& + 
\dfrac{\int_0^\infty \sigma(\nu,T^{n+1}_{i,j})(1-\ie^{-c \sigma(\nu,T^{n+1}_{i,j}) \Delta t/\epsilon^2})b(\nu,T^{n+1}_{i,j}) \dd\nu }{1 - \int_0^\infty \ie^{-c \sigma(\nu,T^{n+1}_{i,j}) \Delta  t/\epsilon^2} b(\nu,T^{n+1}_{i,j})\dd\nu} \\[3mm]
& \times \left(\rho^{n+1}_{i,j}-\sum\limits_{p_k\in\mathcal{P}^{n+1}_{\mathrm{free}}} \dfrac{ w_k }{\Delta x_i \Delta y_j} \chi_{[x_i,x_{i+1}]\times[y_j,y_{j+1}]}(\bx_k(t_{n+1}))\right).
\end{split}
\end{equation}
Note that when $\sigma$ is independent of $\nu$, equation \eqref{eq:source_close} becomes
\begin{equation}
S^{n+1}_{i,j} = \sigma^{n+1}_{i,j} \rho^{n+1}_{i,j},
\end{equation}
which is consistent with the discretization for the gray equation of transfer. 

The macroscopic numerical fluxes across cell interfaces are given by
 \begin{equation}
 \mathcal{F}_{i\pm\frac12,j} = \dfrac{c}{\epsilon \Delta t} \int_{t_n}^{t_{n+1} }\int_0^\infty \int_{\bbS^2} \Omega_x I(t,x_{i\pm\frac12},y_j,\bsOmega,\nu)
 \dd\bsOmega \dd \nu \dd t,
 \end{equation}
 \begin{equation}
 \mathcal{G}_{i,j\pm\frac12} = \dfrac{c}{\epsilon \Delta t} \int_{t_n}^{t_{n+1} }\int_0^\infty \int_{\bbS^2} \Omega_y I(t,x_{i},y_{j\pm\frac12},\bsOmega,\nu)
 \dd\bsOmega \dd \nu \dd t,
 \end{equation}
for which closures are also needed. In constructing a closing relationship for the macroscopic fluxes, we substitute $I$ by a second-order
Taylor expansion of equation \eqref{eq:integral_solution}, 
 \begin{equation}\label{eq:integral_solution_approx_expand}
 I(t,\bx,\bsOmega,\nu) = I^{ini} + I^{an},
 \end{equation}
 where
 \begin{equation}
 I^{ini} = \ie^{-c \sigma (t-t_n) /\epsilon^2} I\left(t_n,\bx-\dfrac{c\bsOmega}{\epsilon} (t-t_n) ,\bsOmega,\nu\right) 
 \end{equation}
 and
 \begin{equation}\label{eq:taylor_solution}
\begin{split}
I^{an} = & \int_{t_n}^t \ie^{-c \sigma(t-s)/\epsilon^2} \dfrac{c \sigma}{\epsilon^2} \left[B(\nu, T(t_{n+1},\bx))
 + \pd{B}{T}\pd{T}{t} (s - t_{n+1}) \right. \\[3mm]
& +\left. \pd{B}{T}\pd{T}{x} \left(-\dfrac{c\Omega_x}{\epsilon}(t-s)\right)+ \pd{B}{T}\pd{T}{y} \left(-\dfrac{c\Omega_y}{\epsilon}(t-s)\right)\right]\dd s.
\end{split}
\end{equation}
Direct computation shows that for interior grids,
 \begin{equation}\label{eq:flux_x}
 \mathcal{F}_{i\pm\frac12,j} = \mathcal{F}^{ini}_{i\pm\frac12,j} + \kappa^{eff}_{i\pm\frac12,j}\pd{U_r}{x}(t_{n+1},x_{i\pm\frac12},y_j),
 \end{equation}
 \begin{equation}\label{eq:flux_y}
 \mathcal{G}_{i,j\pm\frac12} = \mathcal{G}^{ini}_{i,j\pm\frac12} + \kappa^{eff}_{i,j\pm\frac12}\pd{U_r}{y}(t_{n+1},x_i,y_{j\pm\frac12}).
 \end{equation}
In equations \eqref{eq:flux_x} and \eqref{eq:flux_y}, $\mathcal{F}^{ini}_{i\pm\frac12,j}$ and $\mathcal{G}^{ini}_{i,j\pm\frac12}$ are computed during the 
particle tracking process, and their net effect satisfy
\begin{equation}\label{eq:free_transport_flux}
\begin{split}
& \dfrac{\Delta t}{\Delta x_i} \left(\mathcal{F}_{i-\frac12,j}-\mathcal{F}_{i+\frac12,j}\right) + \dfrac{\Delta t}{\Delta y_j}\left(\mathcal{G}_{i,j-\frac12}-\mathcal{G}_{i,j+\frac12}\right)  \\[3mm]
= &  
\sum\limits_{p_k\in\mathcal{P}^{n}_\mathrm{D}+\mathcal{P}^{n+1}_{\mathrm{in}}} \dfrac{w_k}{\Delta x_i \Delta y_j} \chi_{[x_i,x_{i+1}]\times [y_j,y_{j+1}]} (\bx_k(t_{n+1}))  - \sum\limits_{p_k\in\mathcal{P}^{n}_\mathrm{D}}\dfrac{w_k}{\Delta x_i \Delta y_j} \chi_{[x_i,x_{i+1}]\times[y_j,y_{j+1}]}
(\bx_k(t_n)).
\end{split}
\end{equation}
Also, the effective diffusion coefficient is 
 \begin{equation}\label{eq:keff}
 \kappa^{eff} = \int_0^\infty  \dfrac{2\epsilon^2 - c \sigma \Delta t-\ie^{-c\sigma \Delta t/\epsilon^2}(2\epsilon^2+c\sigma \Delta t)}{3 \sigma^2} \left(b(\nu,T^{n+1}) + \dfrac{T^{n+1}}{4} \pd{b}{T} \right) \dd\nu.
 \end{equation}
In equation \eqref{eq:keff}, we first approximate
\begin{equation}
T^{n+1}_{i\pm\frac12,j} = \frac12(T^{n+1}_{i,j} + T^{n+1}_{i\pm 1,j}),\quad
T^{n+1}_{i,j\pm \frac12} = \frac12(T^{n+1}_{i,j} + T^{n+1}_{i,j\pm 1}),
\end{equation}
and then $\kappa^{eff}_{i+\frac12,j}$ is approximated by taking $\sigma$ as $\sigma(\nu, x_i,y_j,T_{i+\frac12,j})$ to produce $\kappa_l$, and then taking $\sigma$ as
$\sigma(\nu, x_{i+1},y_j, T_{i+\frac12,j})$ to produce $\kappa_r$, then take $\kappa^{eff}_{i+\frac12,j} = \frac12(\kappa_l + \kappa_r)$. $\kappa^{eff}$ on other cell boundaries are 
computed analogously.
We approximate $\pd{U_r}{x} $ and $\pd{U_r}{y}$ by central differencing:
 \begin{equation}
 \pd{U_r}{x}(t_{n+1},x_{i+\frac12},y_j) = \dfrac{U^{n+1}_{r,i+1,j}-U^{n+1}_{r,i,j}}{\frac12(\Delta x_{i+1}+\Delta x_i)},
\quad
 \pd{U_r}{x}(t_{n+1},x_{i-\frac12},y_j) = \dfrac{U^{n+1}_{r,i,j}-U^{n+1}_{r,i-1,j}}{\frac12(\Delta x_{i}+\Delta x_{i-1})},
 \end{equation}
  \begin{equation}
 \pd{U_r}{y}(t_{n+1},x_{i},y_{j+\frac12}) = \dfrac{U^{n+1}_{r,i,j+1}-U^{n+1}_{r,i,j}}{\frac12(\Delta y_{j+1}+\Delta y_j)},
\quad
 \pd{U_r}{y}(t_{n+1},x_{i},y_{j-\frac12}) = \dfrac{U^{n+1}_{r,i,j}-U^{n+1}_{r,i,j-1}}{\frac12(\Delta y_{j}+\Delta y_{j-1})}.
 \end{equation}
 So, in summary, $\rho^{n+1}_{i,j}$ is updated by solving
\begin{equation}\label{eq:update_rho}
\begin{split}
& \rho^{n+1}_{i,j} + \dfrac{1}{\Delta x_i}\left( \kappa^{eff}_{i+\frac12,j}\dfrac{U^{n+1}_{r,i+1,j}-U^{n+1}_{r,i,j}}{\frac12(\Delta x_{i+1}+\Delta x_i)}
 - \kappa^{eff}_{i-\frac12,j} \dfrac{U^{n+1}_{r,i,j}-U^{n+1}_{r,i-1,j}}{\frac12(\Delta x_{i}+\Delta x_{i-1})} \right) \\[3mm]
& + \dfrac{1}{\Delta y_j} \left( \kappa^{eff}_{i,j+\frac12}\dfrac{U^{n+1}_{r,i,j+1}-U^{n+1}_{r,i,j}}{\frac12(\Delta y_{j+1}+\Delta y_j)}
 - \kappa^{eff}_{i,j-\frac12} \dfrac{U^{n+1}_{r,i,j}-U^{n+1}_{r,i,j-1}}{\frac12(\Delta y_{j}+\Delta y_{j-1})}\right)  \\[3mm]
= & \rho^n_{i,j} + \sum\limits_{p_k \in \mathcal{P}^n_{\mathrm{D}}+\mathcal{P}^{n+1}_{\mathrm{in}}} \dfrac{w_k}{\Delta x_i \Delta y_j} \chi_{[x_i,x_{i+1}] \times [y_j, y_{j+1}]} (\bx_k(t_{n+1})) \\[3mm] 
& -
 \sum\limits_{p_k \in \mathcal{P}^n_{\mathrm{D}}} \dfrac{w_k}{\Delta x_i \Delta y_j} \chi_{[x_i,x_{i+1}] \times [y_j, y_{j+1}]} (\bx_k(t_n))  
-\dfrac{c \Delta t}{\epsilon^2} (S^{n+1}_{i,j} - \sigma^{n+1}_{p,i,j} U^{n+1}_{r,i,j}). 
\end{split}
\end{equation}
It is solved by coupling with \eqref{eq:discretize_Ur}. 
In \eqref{eq:update_rho} and \eqref{eq:discretize_Ur}, the parameters $\sigma^{n+1}_{i,j}$ and $\beta^{n+1}_{i,j}$ depends 
nonlinearly on the material temperature $T$.
The system is solved by an iteration method similar to that in \cite{sun2018asymptotic} and \cite{shi2020}, where $\sigma$, $\beta$, $C_v$, 
and $\kappa^{eff}$
are lagged in each iteration, and a linear system for $\rho$ and $U_r$ is solved for each iteration step.

\subsection{Particle re-sampling}\label{sec:re-sample}
For re-sampling collision particles at the end of each time step, we employ the same approximation to \eqref{eq:integral_solution_approx}
as  \eqref{eq:integral_solution_approx_firstorder}.
The particles follow the distribution \eqref{eq:integral_solution_discrete}, therefore the weights of re-sampled particles in cell $(i,j)$ add
up to 
\begin{equation}
E^+_{i,j} \Delta x_i \Delta y_j = \rho^{n+1}_{i,j} \Delta x_i \Delta y_j-\sum\limits_{k\in\mathcal{P}^{n+1}_{\mathrm{free}}} w_k \chi_{[x_i,x_{i+1}]\times[y_j,y_{j+1}]}(\bx_k(t_{n+1})).
\end{equation}
Particles are taken to be of equal weight. We specify a reference weight
$w_{ref}$. For cell $(i,j)$, the number of re-sampled particles is
\begin{equation}\label{eq:num_par}
N_{\mathrm{r},i,j}=\left\{\begin{array}{l}
0, \quad \text{if}~E^+_{i,j}\Delta x_i \Delta y_j < w_{ref}, \\[3mm]
\Bigg\lceil \dfrac{E^+_{i,j} \Delta x_i \Delta y_j}{w_{ref}} \Bigg\rceil,\quad \text{else}.
\end{array}\right.
\end{equation}
If $N_{\mathrm{r},i,j}>0$, the weight of each re-sampled particle is assigned
$\dfrac{E^+_{i,j} \Delta x_i \Delta y_j}{N_{\mathrm{r},i,j}}$.

In re-sampling, the distribution of the weights of the re-sampled
particles in space can also be approximated by a piecewise linear reconstruction function
\begin{equation}
E^+(x,y) \approx E^+_{i,j} + s_1 (x-x_i) + s_2 (y - y_j),
\end{equation}
where
\begin{equation}
s_1 = \text{minmod} \left(\dfrac{E^+_{i+1,j}-E^+_{i,j}}{\frac12 (\Delta x_i + \Delta x_{i+1})}, \dfrac{E^+_{i,j}-E^+_{i-1,j}}{\frac12(\Delta x_{i-1}+\Delta x_i)}\right),
\end{equation}
and 
\begin{equation}
s_2 = \text{minmod} \left(\dfrac{E^+_{i,j+1}-E^+_{i,j}}{\frac12 (\Delta y_j + \Delta y_{j+1})}, \dfrac{E^+_{i,j}-E^+_{i,j-1}}{\frac12(\Delta y_{j-1}+\Delta y_j)}\right),
\end{equation}
with the minmod function defined as
\begin{equation}
\text{minmod}(a,b) = \left\{\begin{array}{l}
\text{sign}(a) \text{min}(|a|,|b|),\quad \text{if}~ \text{sign}(a) = \text{sign}(b),\\
0,\quad \text{else}.
\end{array}\right.
\end{equation}
It was proved in \cite{liu2019} that chosing either way of re-sampling 
particle weights does not affect the asymptotic preserving properties of the 
UGKP method. In our numerical simulations, for simplicity, we employ the approximation without reconstruction for the distribution of particle weights in space. 
Furthermore, the re-sampled particles follow an isotropic distribution for the angular variable $\bsOmega$, while their frequency variable follows the distribution
\begin{equation}
 \dfrac{1-\ie^{-c \sigma(\nu,T^{n+1}) \Delta t/\epsilon^2}}{1 - \int_0^\infty \ie^{-c \sigma(\nu,T^{n+1}) \Delta  t/\epsilon^2} b(\nu,T^{n+1})\dd\nu} b(\nu,T^{n+1}).
\end{equation}

In summary, the UGKP algorithm is outlined as follows. Initially, particles are sampled from the given initial condition. Later, for each time step,
the UGKP method consists of the following steps:
\begin{enumerate}
\item Follow the trajectories of each particle before the first collision, and compute the net free transport flux by \eqref{eq:free_transport_flux} for each cell. 
\item Update the macroscopic quantities $\rho^{n+1}$ and $U_r^{n+1}$ solving \eqref{eq:update_rho} and \eqref{eq:discretize_Ur} by iteration. 
Use the closing relationships
\eqref{eq:source_close} for the source term. For each iteration step, after $U_r$ is updated, update $T$ by the relationship \eqref{eq:Ur}.
\item Re-sample the emission particles by \eqref{eq:integral_solution_approx_firstorder}.
\item Sample boundary particles by their respective distribution.
\end{enumerate}

\section{Numerical analysis}
In this section, we perform formal analysis of the asymptotic preserving property and regime adaptive property of the proposed UGKP method, which states that the UGKP method
converges to a five-point central difference scheme for the Rosseland diffusion equation in the 
optically thick regime, and coincides with a Monte Carlo
method in the optically thin regime.
%
\begin{proposition}
For $\epsilon \rightarrow 0$, then the leading order of the specific intensity $I$ of the UGKP method is a Planck function at local 
material temperature,
\begin{equation}
I^{(0)}_{i,j} = B(\nu, T^{n+1}_{i,j}),
\end{equation}
with $T^{n+1}_{i,j}$ satisfying a five-point discretization of the nonlinear diffusion equation
\begin{equation}\label{eq:nonlinear_diff}
C_v \pd{T}{t} + a \pd{}{t} (T^4) = \nabla \cdot \dfrac{a c}{3 \sigma_R} \nabla (T^4).
\end{equation}
\end{proposition}
\begin{proof}
As $\epsilon \rightarrow 0$, we have
\begin{equation}
\ie^{-c\sigma(\nu,T^n)\Delta t /\epsilon^2} I(t_n,x-\frac{c\Omega_x}{\epsilon}\Delta t, y - \frac{c \Omega_y}{\epsilon}\Delta t, \bsOmega,\nu)  
\rightarrow 0.
\end{equation}
and the number of particles $N_p$ also converges to zero as
$\epsilon \rightarrow 0$. Therefore, from equation \eqref{eq:integral_solution_discrete}, 
\begin{equation}
\lim_{\epsilon\rightarrow 0} I^{n+1}_{i,j} =  \lim_{\epsilon \rightarrow 0} \dfrac{1-\ie^{-c \sigma(\nu,T^{n+1}_{i,j}) \Delta t/\epsilon^2}}{1 - \int_0^\infty \ie^{-c \sigma(\nu,T^{n+1}_{i,j}) \Delta  t/\epsilon^2} b(\nu,T^{n+1}_{i,j})\dd\nu} \\[3mm]
\times \dfrac{1}{4\pi}\rho^{n+1}_{i,j} b(\nu,T^{n+1}_{i,j}) = \dfrac{\rho^{n+1}_{i,j}}{4\pi} b(\nu, T^{n+1}_{i,j}).
\end{equation}
From equation \eqref{eq:update_rho}, by matching the orders of $\epsilon$,
\begin{equation}
\lim_{\epsilon \rightarrow 0} S^{n+1}_{i,j} = \sigma^{n+1}_{p,i,j} U^{n+1}_{r,i,j}.
\end{equation}
On the other hand, Eqaution \eqref{eq:source_close} implies that
\begin{equation}
\lim_{\epsilon\rightarrow 0} S^{n+1}_{i,j} = \sigma^{n+1}_{p,i,j} \rho^{n+1}_{i,j}.
\end{equation}
Therefore, in the limit $\epsilon \rightarrow 0$, we have that the leading order of the energy density, $\rho^{n+1,(0)}_{i,j}$, satisfies
\begin{equation}\label{eq:rho_lim}
\rho^{n+1,(0)}_{i,j} = U^{n+1}_{r,i,j}.
\end{equation}
Equation \eqref{eq:rho_lim} implies
\begin{equation}
I^{n+1,(0)}_{i,j} = B(\nu, T^{n+1}_{i,j}).
\end{equation}
Also, from equation \eqref{eq:keff}, we have 
\begin{equation}
\lim_{\epsilon\rightarrow 0 }\kappa^{eff}  = \dfrac{c\Delta t}{3\sigma_R},
\end{equation}
with $\sigma_R$ the Rosseland mean cross-section defined in \eqref{eq:sigmaR}.
Therefore, for $\epsilon \rightarrow 0$, equation \eqref{eq:update_rho} approaches
\begin{equation}\label{eq:update_rho_lim}
\begin{split}
& \dfrac{\rho^{n+1}_{i,j}-\rho^n_{i,j}}{c \Delta t} 
+ \dfrac{1}{3\Delta x_i}\left( \dfrac{1}{\sigma_{R,i+\frac12,j}}\dfrac{U^{n+1}_{r,i+1,j}-U^{n+1}_{r,i,j}}{\frac12(\Delta x_{i+1}+\Delta x_i)}
 - \dfrac{1}{\sigma_{R,i-\frac12,j}}\dfrac{U^{n+1}_{r,i,j}-U^{n+1}_{r,i-1,j}}{\frac12(\Delta x_{i}+\Delta x_{i-1})} \right) \\[3mm]
& + \dfrac{1}{3 \Delta y_j} \left(\dfrac{1}{\sigma_{R,i,j+\frac12}} \dfrac{U^{n+1}_{r,i,j+1}-U^{n+1}_{r,i,j}}{\frac12(\Delta y_{j+1}+\Delta y_j)}
 - \dfrac{1}{\sigma_{R,i,j-\frac12} }\dfrac{U^{n+1}_{r,i,j}-U^{n+1}_{r,i,j-1}}{\frac12(\Delta y_{j}+\Delta y_{j-1})}\right)  \\[3mm]
= & -\dfrac{1}{\epsilon^2} (S^{n+1}_{i,j} - \sigma^{n+1}_{p,i,j} U^{n+1}_{r,i,j}). 
\end{split}
\end{equation}
Dividing \eqref{eq:discretize_Ur} on both sides by $\epsilon^2 \beta^{n+1}_{i,j}$, we have
\begin{equation}\label{eq:update_Ur_lim}
C^{n+1}_{v,i,j} \dfrac{1}{\beta^{n+1}_{i,j}} \dfrac{U^{n+1}_{r,i,j}-U^n_{r,i,j}}{\Delta t} =\dfrac1{\epsilon^2} \left(S^{n+1}_{i,j}-\sigma^{n+1}_{p,i,j} U^{n+1}_{r,i,j}\right).
\end{equation}
Adding \eqref{eq:update_rho_lim} and \eqref{eq:update_Ur_lim}, and making use of the relationships \eqref{eq:rho_lim} and 
\begin{equation}
\pd{U_r}{T} = 4acT^3 = \beta,
\end{equation}
we have that $T^{n+1}_{i,j}$ satisfy a five-point discretization of the nonlinear diffusion equation \eqref{eq:nonlinear_diff}.
\end{proof}

\begin{proposition}
In the optically thin regime as $\sigma \rightarrow 0$, the UGKP method is consistent to the Monte Carlo method for the collisionless radiative transfer 
equation.
\end{proposition}
\begin{proof}
Direct computation shows when taking the limit $\sigma \rightarrow 0$, we have $\kappa^{eff} \rightarrow 0$. Therefore, \eqref{eq:update_rho} becomes
\begin{equation}\label{eq:update_rho_discrete}
\begin{split}
 \rho^{n+1}_{i,j} 
=  \rho^n_{i,j}  & + \sum\limits_{p_k\in\mathcal{P}^n_{\mathrm{D}}+\mathcal{P}^{n+1}_{\mathrm{in}}} \dfrac{w_k}{\Delta x_i \Delta y_j} \chi_{[x_i,x_{i+1}] \times [y_j, y_{j+1}]} (\bx^k(t_{n+1}))  \\[3mm]
 & -
 \sum\limits_{p_k\in\mathcal{P}^n_{\mathrm{D}}} \dfrac{w_k}{\Delta x_i \Delta y_j} \chi_{[x_i,x_{i+1}] \times [y_j, y_{j+1}]} (\bx_k(t_n)). 
 \end{split}
\end{equation}
At the same time, the specific intensity evolves by
\begin{equation}
I(t,\bx,\bsOmega,\nu) =  I\left(0,\bx-\dfrac{c\bsOmega}{\epsilon} t ,\bsOmega,\nu\right).
\end{equation}
Therefore, the UGKP method is consistent with the Monte Carlo solver for this case.
\end{proof}

\begin{proposition}
The computational complexity of the UGKP method adapts to the flow regime. 
In the optically thick regime as $\sigma \rightarrow \infty$, the computational complexity of the UGKP method is consistent with the diffusive scheme.
In the optically thin regime as $\sigma \rightarrow 0$, the computational complexity of the UGKP method is consistent with the Monte Carlo scheme.
\end{proposition}
\begin{proof}
In the UGKP method, the computational cost comprises of the macroscopic computational cost of solving the diffusion system $N_{\text{DIFF}}$ and the microscopic computational cost of particle tracking $N_{\text{MC}}$.
Since the particles are tracked until their first collision event or until
the end of the time step, the tracking time of each particle is $\text{min}\{\nu^{-1},\Delta t\}$, where $\nu=c\sigma\epsilon^{-2}$ is the collision rate.
Assume that the numerical cell size is $\Delta x$, the total computational cost of particle tracking within one time step is
\begin{equation}
N_{\text{MC}}\sim \text{min}\{N_pc\nu^{-1}\Delta x^{-1},N_pc\Delta t\Delta x^{-1}\},
\end{equation}
where $N_p$ is the total number of particles.
In the optically thick regime, the UGKP's computational cost within one 
time step follows
\begin{equation}
\begin{aligned}
\lim_{\sigma \to \infty}(N_{\text{DIFF}}+N_{\text{MC}})=&N_{\text{DIFF}}+\lim_{\nu \to \infty}N_{\text{MC}}\\
\sim&N_{\text{DIFF}},
\end{aligned}
\end{equation}
which shows the UGKP's computational cost is similar to the diffusive scheme in the optically thick regime.
In the optically thin regime, the UGKP's computational cost within one time step follows
\begin{equation}
\begin{aligned}
\lim_{\sigma \to 0}(N_{\text{DIFF}}+N_{\text{MC}})=&N_{\text{DIFF}}+\lim_{\nu \to 0}N_{\text{MC}}\\
=&N_{\text{DIFF}}+N_pc\Delta t\Delta x^{-1}.
\end{aligned}
\end{equation}
Considering the particle number $N_p$ is large, the UGKP's computational cost is similar to the Monte Carlo method in the optically thin regime.
\end{proof}
\section{Numerical results}

In this section, we present six numerical examples to validate the capability and efficiency of the proposed frequency-dependent UGKP method.
The units of all our simulation are as follow: the unit of length is taken to be centimeter (cm), the unit of time is nanosecond (ns), the unit of 
temperature is kilo electron-volt (keV), and the unit of energy is $10^9$ Joules (GJ). Under the above units, the speed of light is $29.98 \mathrm{cm/ns}$, 
and the radiation constant is $0.01372 \mathrm{GJ/cm^3/keV^4}$.  In all our examples we take $\epsilon = 1$.
All computations are performed for a sequential code on a computer with 
a CPU frequency of 2.2 GHz.

\subsection{Marshak wave problem for the gray equation of transfer}
In this example, we test the classical Marshak wave problem, which describes the propagation 
of a thermal
wave in a slab. The slab is initially cold, and is heated by a constant source incident on its boundary. This benchmark problem is also studied in \cite{Larsen2013,
sun2015asymptotic1,steinberg2022multi}. We define $T_{\mathrm{keV}} = 1 \mathrm{keV}$, and take the absorption/emission coefficient to be 
$\sigma(\nu,T) = \dfrac{300}{(T/T_{\mathrm{keV}})^3} \text{cm}^{-1}$. The specific heat capacity is set to be $0.3 \text{GJ}/\text{cm}^{3}/\text{keV}$. 
The initial material temperature $T$ is $10^{-3}
\text{keV}$. Initially, material and radiation energy are at equilibrium. The specific intensity on the left boundary is kept at a constant distribution. The distribution on the left boundary is isotropic in angular variable, and is a Planckian distribution associated with a temperature of  $1 \mathrm{keV}$ concerning
the frequency variable. 
The computation
domain is $[0,\infty)$ but taken to be $L=[0,0.5]$ in the simulations. The UGKP method uses $200$ uniform cells in space and takes a time step of 
$\Delta t  = 1.6\times 10^{-3} \mathrm{ns} $. 
\begin{figure}[htbp]
  \centering 
  \includegraphics[width=0.5\textwidth,
  trim={15mm 90mm 20mm 90mm}, clip]{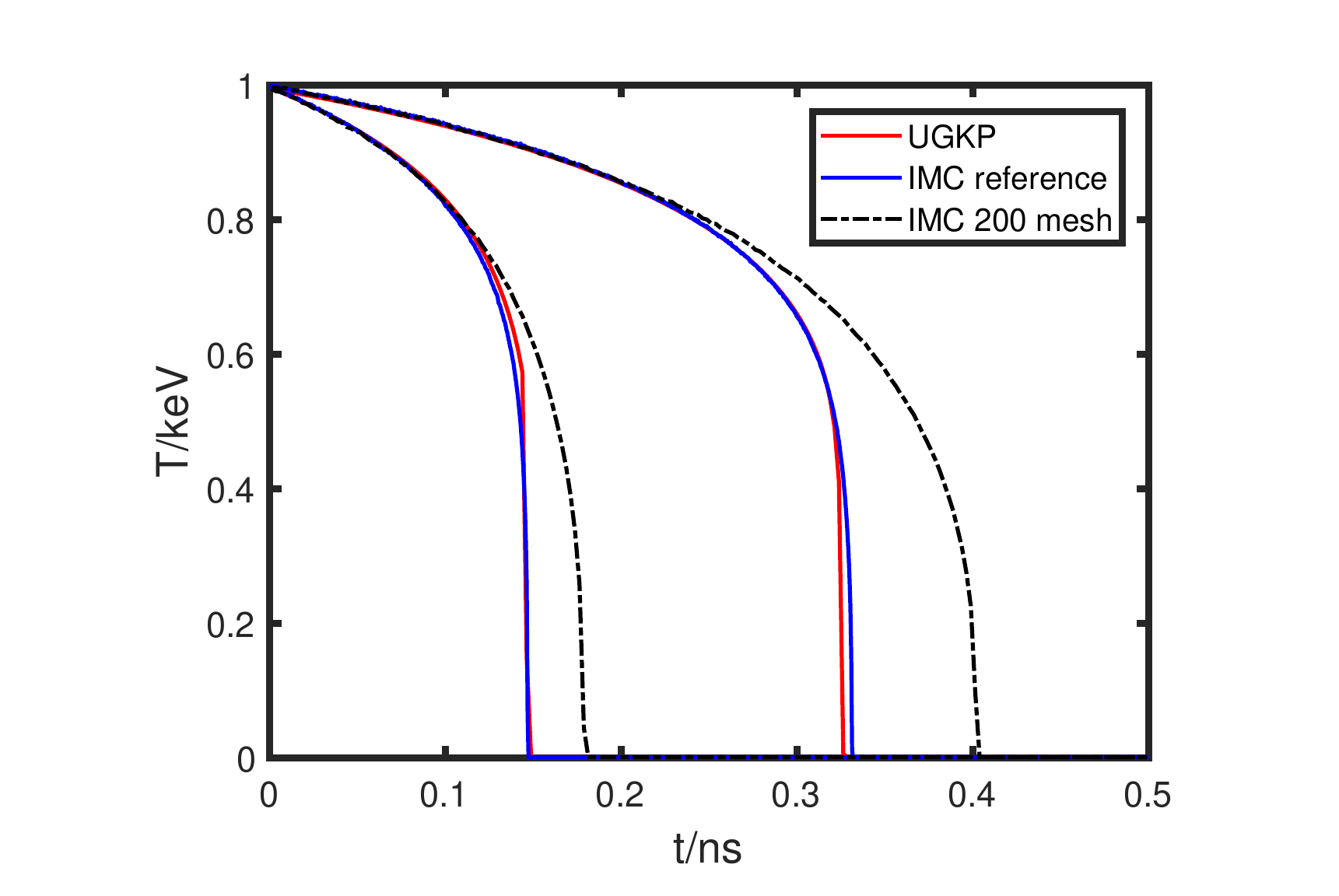}
  \caption{The material temperature $T$ of the Marshak Wave problem for the gray equation of transfer 
  at time $t = 10 \mathrm{ns}$ and $50 \mathrm{ns}$.}
  \label{fig:marshak300}
\end{figure}

In Figure \ref{fig:marshak300}, the numerical results of the material temperature at time $t = 10 \mathrm{ns}$ and $50 \mathrm{ns}$ are plotted. 
The reference solutions are obtained by the 
implicit Monte Carlo method using $1000$ uniform cells in space. The wavefronts of the UGKP method are 
consistent with those 
of the reference solution. Figure \ref{fig:marshak300} also presents the results obtained by the implicit Monte Carlo method
using $200$ uniform spatial cells, and it could be seen that for this mesh size, the wavefront of the implicit Monte Carlo method travels 
faster than the reference solution due to teleportation error. However, the UGKP method does not suffer 
from the teleportation error because of its asymptotic preserving properties. Therefore, it can
use a much larger cell size and time step than the particle mean free path and collision time.
For the computation time, the IMC reference solution takes $186$ hours 
to reach the simulation time of 
$50 \mathrm{ns}$, while the UGKP method takes $8.5$ hours, which demonstrates the UGKP method is more efficient than IMC for this example. 
The improved efficiency of the UGKP method is because it does not need to compute the large amount 
of effective scattering employed by the IMC method for this optically thick scenario.


\subsection{Frequency-dependent homogeneous problem}
The second example we consider is a frequency-dependent homogeneous problem. Initially, radiation and material temperature are equal at 
$1 \mathrm{keV}$, but the specific intensity is not a Planck distribution for frequency. Instead,
the initial specific intensity is 
\begin{equation}
I_0(\bx,\bsOmega,\nu) = \dfrac{a c T_r^4}{4\pi} b(\nu, T_0),
\end{equation}
with $T_0 = 0.1 \mathrm{keV}$ and $T_r = 1 \mathrm{keV}$, and $b(\nu,T_0)$ is the normalized Planck distribution as 
defined in \eqref{eq:normalized_planck}. The computation domain is $[0,0.01] \times [0,1]$ with periodic boundary conditions
imposed on all boundaries. For this setup, the solution remains constant in space during any arbitrary 
simulation period, but there is energy exchange between radiation and background material. 
We take $\sigma$ to be 
\begin{equation}
\sigma(\nu,T) = \left\{\begin{array}{l}
10^{-8} \mathrm{cm}^{-1}, \quad \text{if}~h \nu \in (0,1),\\[3mm]
1000 \mathrm{cm}^{-1}, \quad\text{else},
\end{array}
\right.
\end{equation}
and the specific heat capacity is set as $0.3 \text{GJ}/\text{cm}^{3}/\text{keV}$. Because the absorption
coefficient $\sigma$ is almost zero for the low-frequency range, there is little energy exchange between 
radiation and matter for this range, and the specific intensity is expected to remain close to the 
initial distribution. For the rest of the frequency range, $\sigma$ is very large, therefore one would expect
the final specific intensity to be close to the equilibrium distribution, which is a Planck function 
associated with the final material temperature. So, we expect the distribution of the specific intensity 
concerning the frequency variable to evolve into two peaks. This distribution of the specific intensity is
a phenomenon unique to the multi-frequency setup, and cannot be captured in the gray approximation to the 
equation of transfer.
In our simulations, a uniform time step of $2.6\times10^{-4}\mathrm{ns}$ is taken and one mesh grid 
in space is used.
\begin{figure}[htbp]
  \centering 
  \includegraphics[width=0.5\textwidth,
  trim={20mm 80mm 20mm 80mm}, clip]{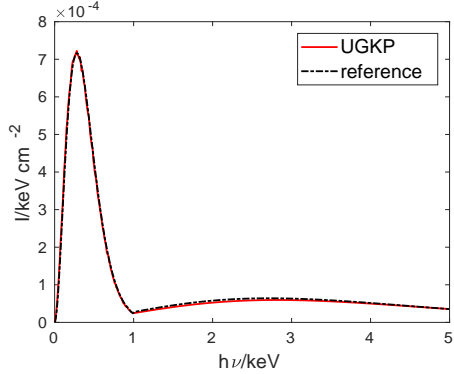}
  \caption{The specific intensity $I$ of the frequency-dependent homogeneous problem at time $t = 1 \mathrm{ns}$.}
  \label{fig:homogeneous}
\end{figure}

In Figure \ref{fig:homogeneous}, the solution of the specific intensity at $t = 1 \mathrm{ns}$
with respect to the frequency variable is plotted. The UGKP solution is compared
with a reference solution which directly computes the homogeneous problem using multigroup discretization of 
the frequency variable with $10000$ groups. 
The UGKP solution and the reference solution are almost identical, 
showing the capability of the UGKP method to accurately compute 
frequency-dependent energy exchange. The specific intensity displays two peaks. The left peak is determined by the initial distribution, and the right peak is determined by the final material temperature.
The UGKP solution agrees well with our intuition, further verifying the accuracy of 
this method.

\subsection{Frequency-dependent Marshak wave problem}
In the third example taken from \cite{sun2015asymptotic2}, we consider another Marshak wave problem, but where the opacity depends on frequency. We take the absorption/emission coefficient to be of the form
\begin{equation}
\sigma(\nu,T) = \dfrac{1000}{(h\nu)^3 \sqrt{kT}},
\end{equation}
and set the specific heat capacity to be constant at $0.1 \mathrm{GJ}/\mathrm{keV}/\mathrm{cm}^3$. Initially, the material temperature is $T_0 = 10^{-3} \mathrm{keV}$, and the 
specific intensity is
\begin{equation}
I_0(\bx,\bsOmega,\nu) = B(\nu, T_0).
\end{equation}
The left boundary is kept at a constant specific intensity. Its distribution 
with respect to the angular
variable is isotropic. With respect to the frequency variable, the specific intensity
is 
a Planckian associated with a temperature of $1 \mathrm{keV}$. A reflective boundary condition is imposed
on the right. This example tests an algorithm's capability to handle cases where the absorption coefficient varies 
widely for the whole frequency range. For low frequency where $ h \nu$ is small, $\sigma$ can be very 
large, therefore the system is optically thick. On the other hand, $\sigma$ drops quickly as $h\nu$ increases,
and the system becomes optically thin
for $h \nu$ surpassing $10 \mathrm{keV}$. This multiscale property makes 
this example a challenging simulation problem.

\begin{figure}[htbp]
  \centering 
  \subfloat[Radiation temperature.]{
  \includegraphics[width=0.48\textwidth,
  trim={15mm 80mm 20mm 80mm}, clip]{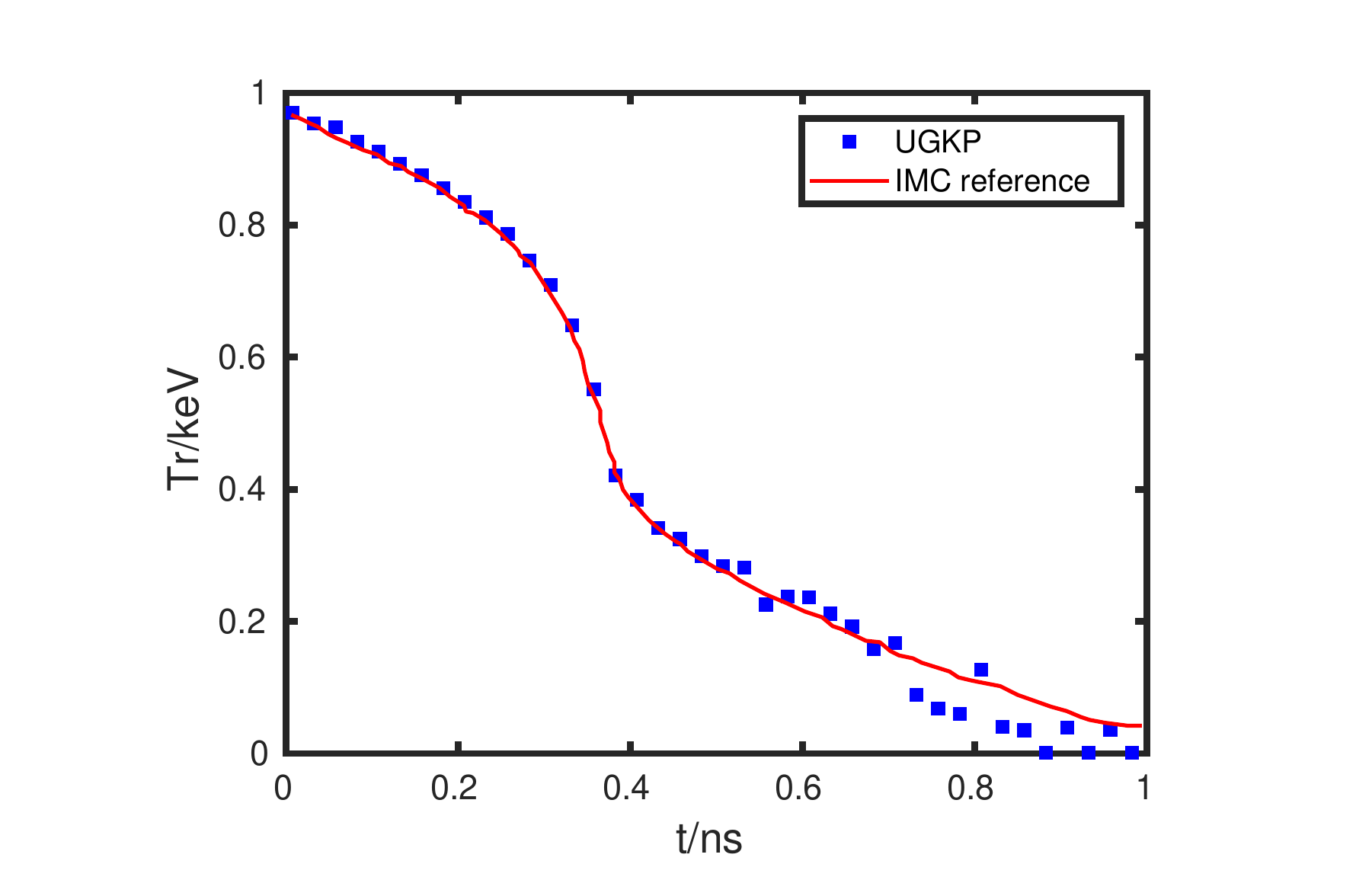} 
}
  \subfloat[Material temperature.]{
 \includegraphics[width=0.48\textwidth,
  trim={15mm 80mm 20mm 80mm}, clip]{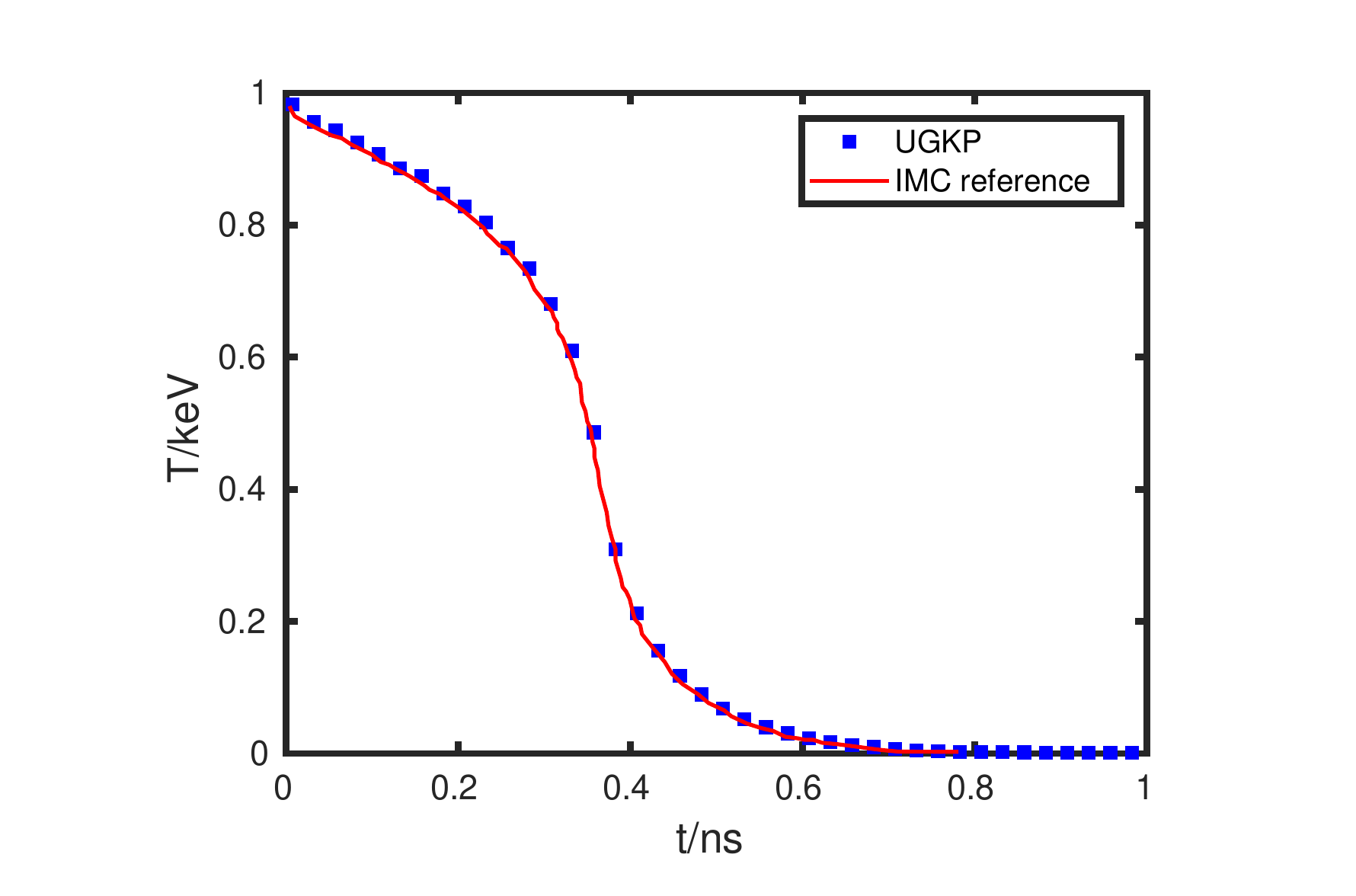} 
}
  \caption{The radiation and material temperature of the frequency-dependent Marshak Wave problem at time $t = 1 \mathrm{ns}$.}
  \label{fig:mf_marshak}
\end{figure}

In our simulations, the computation domain is $[0 \mathrm{cm}, 5 \mathrm{cm}]$ and the UGKP method uses a uniform mesh size of $\Delta x = 0.005 
\mathrm{cm}$. The time step is $1.3\times 10^{-4} \mathrm{ns}$.
Figure \ref{fig:mf_marshak}
presents the radiation and material temperatures computed by the UGKP method at the simulation time of $1 \mathrm{ns}$ 
and compares them with the IMC solution provided in 
\cite{sun2015asymptotic2}. The radiation temperature is defined as
\begin{equation}
T_r = \sqrt[4]{\dfrac{\rho}{a c}}.
\end{equation}
A good agreement has been obtained between the UGKP and IMC reference solutions for 
both the radiation and material temperatures. For this example, the IMC solution in \cite{sun2015asymptotic2} takes $1344$ minutes
to reach the simulation time of $1 \mathrm{ns}$, while our implementation of the UGKP method takes $349$ minutes. 
Therefore, the UGKP method is more efficient than
the IMC method for this case.

\subsection{Marshak wave problem with heterogeneous opacity}
The fourth example we consider consists a jump in the absorption coefficient.
Within a computation domain of $[0 \mathrm{cm}, 3 \mathrm{cm}]$, the 
function $\sigma(\nu,T)$ takes the following form:
\begin{equation}
\sigma(\nu,T) = \left\{\begin{array}{l}
\dfrac{10}{(h\nu)^3 \sqrt{kT}}, \quad \text{if}~ 0 \mathrm{cm} \leq x \leq 2 \mathrm{cm},\\[3mm]
\dfrac{1000}{(h\nu)^3 \sqrt{kT}}, \quad\text{else}.
\end{array}
\right.
\end{equation}
Therefore, at any arbitrary frequency, there is a discontinuity at the material interface of
$x = 2\mathrm{cm}$. The radiation always travels through a relatively optically thin region
to a relatively optically thick one.
The specific heat capacity is held constant at $0.1 \mathrm{GJ}/\mathrm{keV}/\mathrm{cm}^3$.  The initial material temperature is 
$T_0 = 10^{-3} \mathrm{keV}$. The initial specific intensity is isotropic with respect to the angular variable and a Planckian at $T_0$ with respect to 
frequency. 
The specific intensity on the left boundary is kept constant at an isotropic angular distribution and  
a Planckian frequency distribution associated with a temperature of $1 \mathrm{keV}$. A reflective boundary condition is imposed
on the right boundary. This example tests a method's ability to treat discontinuity in opacity.
\begin{figure}[htbp]
  \centering 
  \subfloat[Radiation temperature.]{
  \includegraphics[width=0.48\textwidth,
  trim={15mm 80mm 20mm 80mm}, clip]{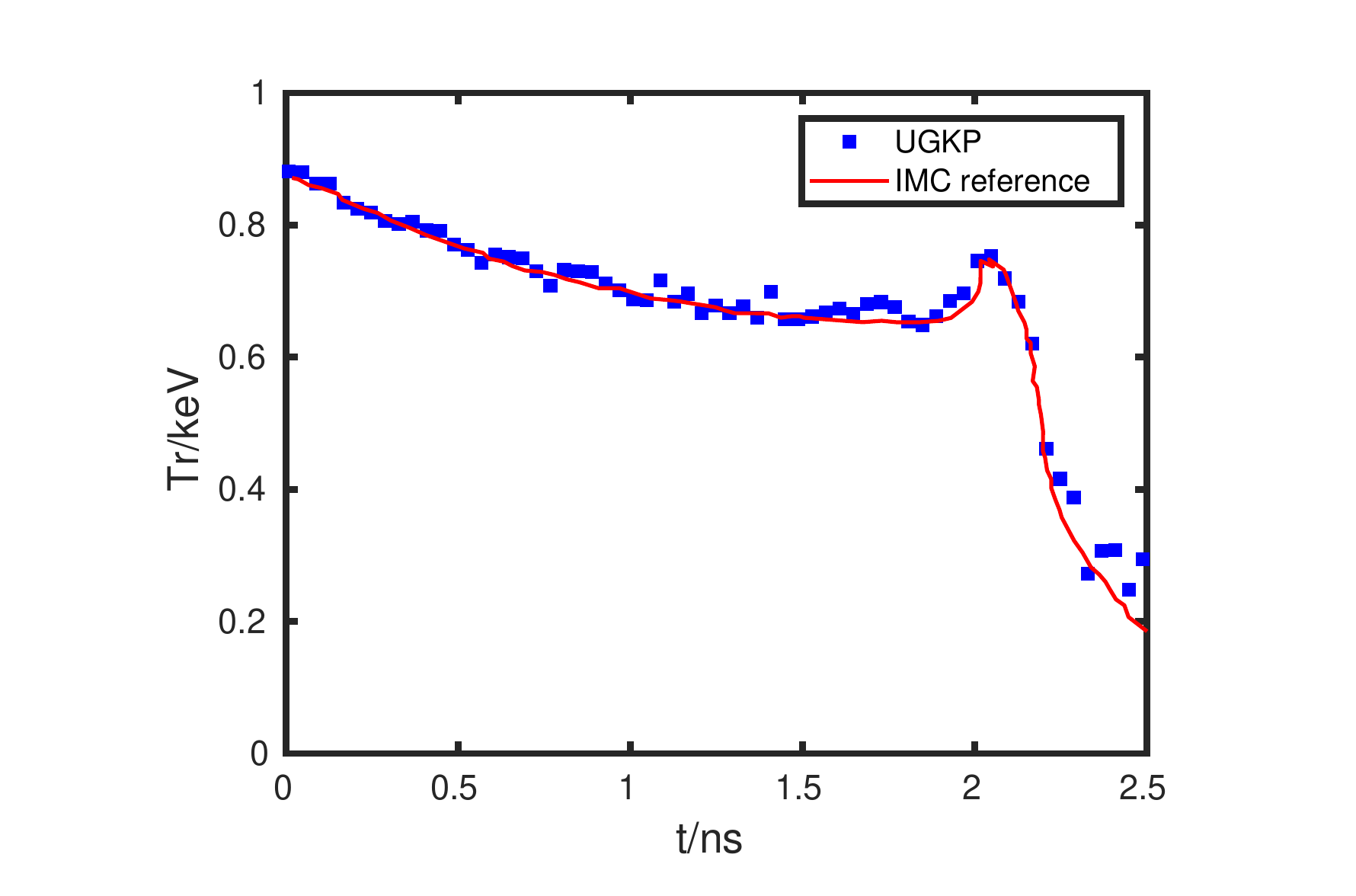} 
}
  \subfloat[Material temperature.]{
 \includegraphics[width=0.48\textwidth,
  trim={15mm 80mm 20mm 80mm}, clip]{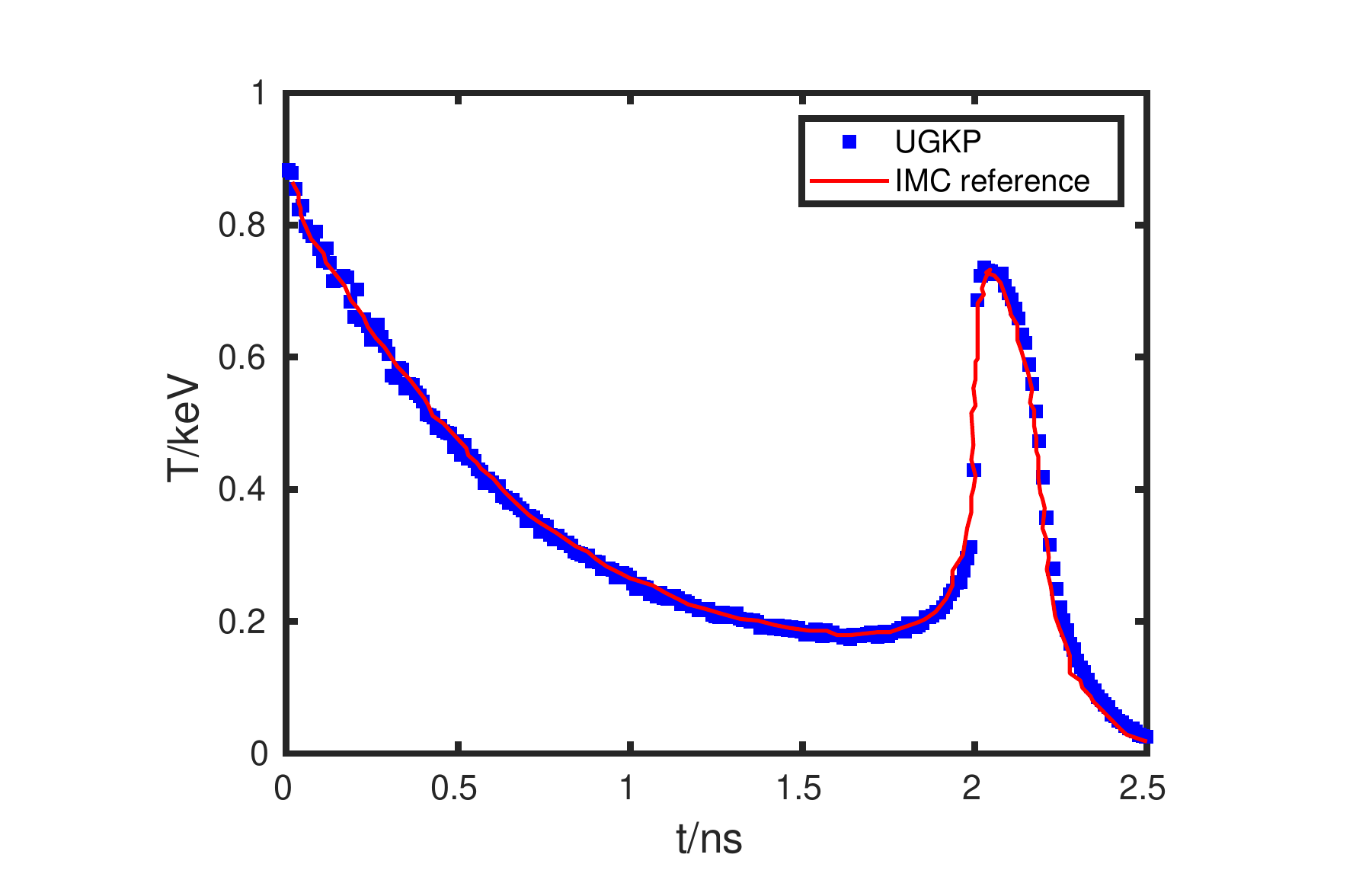} 
}
  \caption{The radiation and material temperature of the frequency-dependent Marshak Wave problem with heterogeneous opacity 
  at time $t = 1 \mathrm{ns}$.}
  \label{fig:ht_marshak}
\end{figure}

The UGKP method uses a uniform mesh size of $\Delta x = 0.005 
\mathrm{cm}$. The time step is taken to be $1.3 \times 10^{-4} \mathrm{ns}$. Figure \ref{fig:ht_marshak} compares the material and radiation temperatures computed by the UGKP method 
with the IMC solution provided in 
\cite{sun2015asymptotic2} at the simulation time of $1 \mathrm{ns}$, and good agreement could be observed. For this example, the IMC solution in \cite{sun2015asymptotic2} takes $363$ minutes
to reach the simulation time of $1 \mathrm{ns}$, while our implementation of the UGKP method takes $121$ minutes, 
showing the UGKP method to be 
more efficient than
the IMC method for this case.

\subsection{A hohlraum problem for the gray equation of transfer}\label{sec:grayHohlraum}
For the fifth example, we study the hohlraum problem for the gray equation of transfer.
This problem describes the heating of a cavity by a radiation source.
It is of wide interest in the literature and has also been 
studied in \cite{mcclarren2010robust, steinberg2022multi}. 
\begin{figure}[htbp]
  \centering 
  \includegraphics[width=0.5\textwidth]{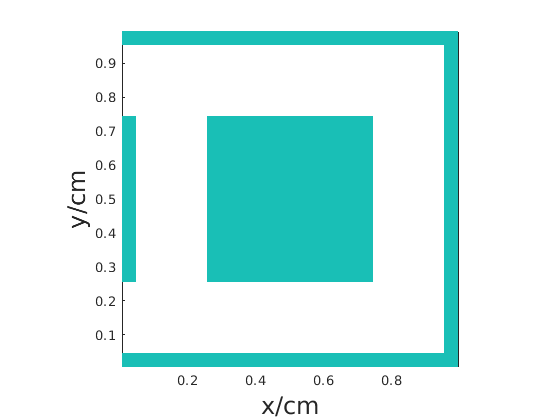}
  \caption{The layout of the hohlraum problem for the gray equation of transfer. The green regions are where $(x,y)\in[0,0.05] \times [0.25,0.75]$, and $(x,y) \in [0.25, 0.75]\times [0.25,0.75]$, $(x,y) \in [0,1]\times [0,0.05]$, $(x,y) \in [0,1]
  \times [0.95,1]$ and $(x,y) \in [0.95,1] \times [0,1]$.}
  \label{fig:grayHohlraum_setup}
\end{figure}

The layout of the problem is shown in Figure \ref{fig:grayHohlraum_setup}. 
The white region within the computation domain of $[0 \mathrm{cm}, 1 \mathrm{cm}] 
\times [0 \mathrm{cm}, 1\mathrm{cm}]$ is a vacuum. In our simulation, we take an absorption 
coefficient of $\sigma = 10^{-8} \mathrm{cm}^{-1}$ and a specific heat capacity of $C_v = 10^{-4} \mathrm{GJ}/\mathrm{keV}/\mathrm{cm}^3$ for this region. The green
regions are filled with material satisfying $\sigma = \dfrac{100}{(T/T_{\mathrm{keV}})^3} \mathrm{cm}^{-1}$, where $T_{\mathrm{keV}} = 1\mathrm{keV}$. The
specific heat capacity $C_v$ for the green regions are $0.3 \mathrm{GJ}/\mathrm{keV}/\mathrm{cm}^3$.
The initial material temperature is uniform at $10^{-3} \mathrm{keV}$. The radiation and material temperatures
are initially at equilibrium. The initial specific intensity is isotropic in the angular variable and a Planckian
in the frequency variable. 
The entire left boundary is kept constant with an angularly 
isotropic specific intensity of $1 \mathrm{keV}$ black body source. The reflective boundary is imposed on the
right. The upper and lower boundaries are kept constant at a specific intensity which is a Planckian 
corresponding to a temperature of $10^{-3} \mathrm{keV}$.

For this problem, the material is initially cold and therefore optically thick. Its opacity decreases as radiation heat up the material. Therefore,
the absorption coefficient varies over a wide range in time and space. The presence of a vacuum in the computation domain further complicates the 
problem and makes it even more challenging. It was shown \cite{mcclarren2010robust} that for this problem, the diffusion approximation could not capture the correct physics, and therefore it is necessary
to simulate the original radiative transfer equation \eqref{eq:rte}. 
\begin{figure}[htbp]
  \centering 
  \subfloat[Radiation temperature.]{
  \includegraphics[width=0.48\textwidth]{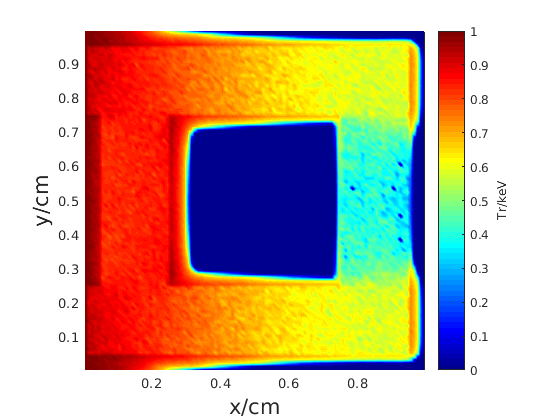}
}
 \subfloat[Material temperature.]{
 \includegraphics[width=0.48\textwidth]{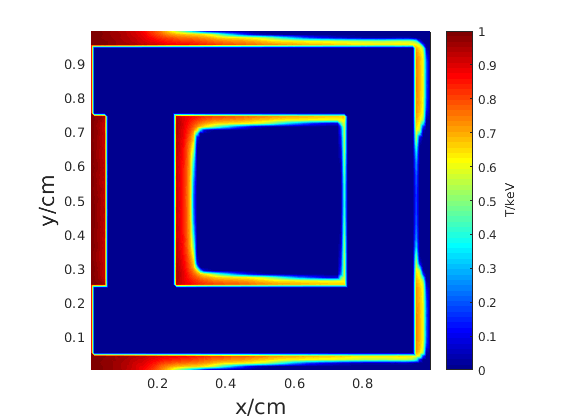}
}
  \caption{The radiation and material temperatures of the hohlraum problem for the gray equation of transfer 
  at time $t = 1 \mathrm{ns}$.}
  \label{fig:grayHohlraum}
\end{figure}

The UGKP method uses a uniform spatial mesh of $100 \times 100$, and the time step is taken to be $2.7\times 10^{-4} \mathrm{ns}$. Figure \ref{fig:grayHohlraum} presents the contour plots of the UGKP solution for the radiation temperature and the material temperature at the simulation time $1 \mathrm{ns}$. Our solutions are in 
general agreement with the solutions presented in \cite{mcclarren2010robust, steinberg2022multi}. Also,
it has been stated in \cite{mcclarren2010robust} that for this problem, the solution should preserve two properties: firstly, there should 
be a non-uniform heating of the central block.
Secondly, there should be less radiation directly behind the block than those regions within sight of the source. The UGKP solutions satisfy these 
two properties, verifying the accuracy of our method.

\subsection{Frequency-dependent hohlraum problem}
The final example we present is the hohlraum problem for the multi-frequency radiative transfer equation. The setup of this problem is the same as that studied in
\cite{Hammer2019}. The layout of this problem is shown in Figure 
\ref{fig:mfHohlraum_setup}. The computation domain is $[0 \mathrm{cm}, 0.65 \mathrm{cm}] \times [0 \mathrm{cm},
1.4 \mathrm{cm}]$. The white region within the computation domain is almost a vacuum, and we take the absorption coefficient to be $\sigma = 10^{-8}
\mathrm{cm}^{-1}$ and the specific heat capacity $C_v$ to be $10^{-4} \mathrm{GJ}/\mathrm{keV}/\mathrm{cm}^3$.  The green regions have a frequency-dependent opacity with the form
\begin{equation}
\sigma(\nu, T) =  \dfrac{1000(1-\ie^{-\frac{h\nu}{kT}})}{(h\nu)^3}.
\end{equation}
The specific heat capacity within the green regions is $C_v = 0.3 \mathrm{GJ}/\mathrm{keV}/\mathrm{cm}^3$. The initial material temperature is $T_0 = 10^{-3} \mathrm{keV}$, and radiation and material are initially at equilibrium. Therefore, the initial specific intensity is
\begin{equation}
I_0(\bx,\bsOmega,\nu) = B(\nu,T_0).
\end{equation}
The reflective boundary condition is imposed on the left boundary. The lower boundary is kept constant with an angularly isotropic specific intensity of 
$0.3 \mathrm{keV}$ black body source. The upper and right  boundaries are kept constant at a specific intensity which is a Planckian corresponding to a temperature of $10^{-3} \mathrm{keV}$.

\begin{figure}[htbp]
  \centering 
  \includegraphics[width=0.25\textwidth]{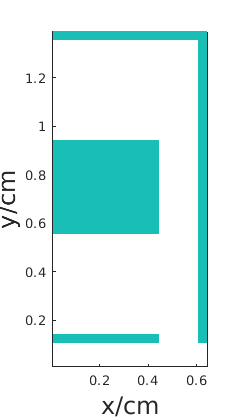}
  \caption{The layout of the hohlraum problem for the frequency-dependent equation of transfer. The green regions are where $(x,y)\in[0,0.45]
  \times [0.1,0.15]$, and $(x,y) \in [0, 0.45]\times [0.55,0.95]$, $(x,y) \in [0.6,0.65]\times [0.1,1.4]$, 
  and $(x,y) \in [0,0.65] \times [1.35,1.4]$.}
  \label{fig:mfHohlraum_setup}
\end{figure}

This problem is even more complicated than the hohlraum problem for the gray equation of
transfer studied in Section \ref{sec:grayHohlraum}, as the material opacity depends on frequency. For low frequency, the 
material is optically thick, but for the higher frequency it becomes optically thin, and the opacity varies widely within the 
whole frequency range.
The presence of the vacuum in the computation domain and the fact the opaque walls within the cavity are sometimes 
very narrow further complicates the 
problem, making it a very challenging benchmark. 
\begin{figure}[htbp]
  \centering 
  \subfloat[Radiation temperature.]{
  \includegraphics[width=0.48\textwidth]{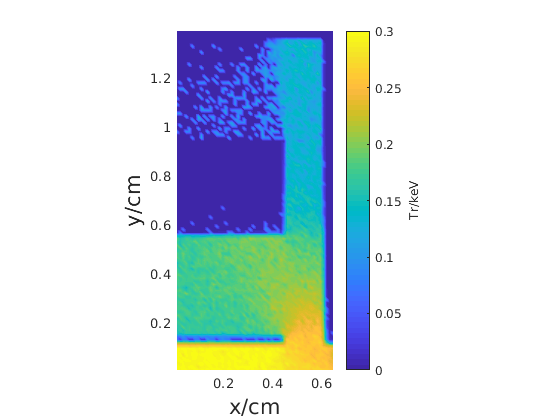}
}
  \subfloat[Material temperature.]{
 \includegraphics[width=0.48\textwidth]{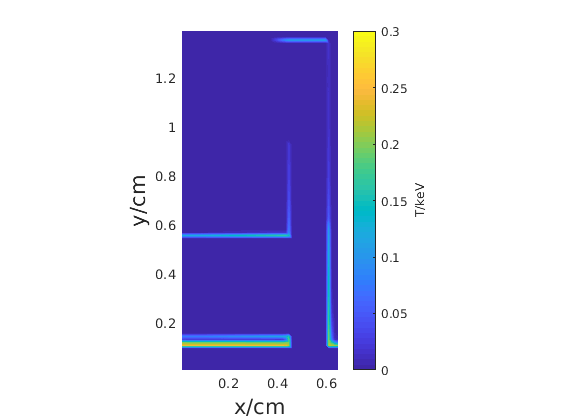}
}
  \caption{The radiation and material temperature of hohlraum problem for the frequency-dependent equation of transfer 
  at time $t = 10 \mathrm{ns}$.}
  \label{fig:mfHohlraum}
\end{figure}

In computing this example, a limiter is employed to suppress artificial 
oscillations at material interfaces. Based on similar considerations as 
\cite{xu2001gas}, we define
\begin{equation}\label{eq:tau}
\tau = \dfrac{\epsilon^2}{c \sigma} + 10 \Delta t\dfrac{T_l-T_r}{T_l+T_r}
\dfrac{2}{\Delta x_l + \Delta x_r}
\end{equation}
at material interfaces. In \eqref{eq:tau}, $T_l$, $T_r$, and $\Delta x_l$,
$\Delta x_r$ are the left and right values of material temperature and grid
length on both sides of the interface. In computing the distance to absorption
for particles crossing the interface, equation \eqref{eq:col_dis} is replaced
by 
\begin{equation}
d_{\mathrm{COL}} = c |\ln \xi| \tau.
\end{equation}
For computing the macroscopic solver, $\kappa^{eff}$ in equation 
\eqref{eq:keff} is multiplied by a factor of 
\begin{equation}
f = \dfrac{1- \exp\left(-\Delta t/\tau\right)}{1-\exp\left(-c \sigma \Delta t/\epsilon^2\right)}.
\end{equation}
We note that this limiter is only applied to material interfaces where there 
is a large jump in material temperature, and does not affect the asymptotic 
preserving properties of the UGKP method.

The UGKP method uses a uniform spatial mesh of $52 \times 112$, and the time step is taken to be $1.7\times 10^{-4} \mathrm{ns}$. The contour plots of the UGKP solution for 
the radiation temperature and the material temperature at the simulation time $10 \mathrm{ns}$ are shown in Figure \ref{fig:mfHohlraum}. The UGKP solutions are in rough 
agreement with the solutions presented in \cite{Hammer2019}, verifying the accuracy of our method. Also,
we observe a smooth material temperature profile, both at the lower wall and the lower side of the center block, which shows that unlike the 
deterministic particle method discussed in \cite{Hammer2019}, the UGKP method does not suffer from ray effects. 
The UGKP solution for the radiation temperature also does not show 
ray effects \cite{Morel2003Ray}, though the solution is noisier than the solutions presented in \cite{Hammer2019}.


\section{Conclusions}
\label{sec:conclusion}
We have extended the unified gas-kinetic particle method (UGKP) studied previously in \cite{liweiming2020, shi2020}
to multi-frequency radiative transfer. Particle sampling combined with analytical representation of the specific intensity's
reliance on frequency enables our method to preserve the Rosseland diffusion limit. Compared with the implicit Monte
Carlo method, our method does not track particle trajectory after first collision event, therefore a high efficiency is obtained.
Numerical analysis shows that our UKGP method is asymptotic preserving for both free-streaming and diffusion limit. We 
demonstrate with numerical examples the accuracy and efficiency of the proposed UGKP method for benchmark problems
where there is a wide variation of opacity.
Future works includes extension to arbitrary quadrilateral mesh and cylindrical geometries.

\section*{Acknowledgements}

We thank Dr. Yajun Zhu from HKUST for helpful discussions. The authors are
partially supported by National Key R\&D Program of China (2022YFA1004500).
Weiming Li is partially supported by the National Natural Science 
Foundation of China (12001051). Chang Liu is partially supported by the National Natural Science Foundation of China (12102061). Peng Song is partially supported by the National Natural Science Foundation of China (12031001).


\end{document}